\documentclass[twoside,11pt]{article}

\usepackage[latin1]{inputenc}
\usepackage[T1]{fontenc}
\usepackage[english]{babel}
\usepackage{color}
\usepackage{tocloft}
\usepackage{verbatim}

\usepackage{tikz}
\usetikzlibrary{snakes}
\usetikzlibrary{arrows}
\usepackage{tikz-cd}

\usepackage[
  hmarginratio={1:1},     % equal left and right margins
  vmarginratio={1:1},     % equal top and bottom margins
  textwidth=17cm,        % new text width
  textheight=24cm,
  heightrounded          % always useful
]{geometry}

\usepackage{amsmath,amsthm,amssymb}
\usepackage{mathrsfs, mathtools}
\usepackage{dsfont}

\allowdisplaybreaks

\usepackage[all]{xy} 
\SelectTips{cm}{11} % Makes all xy arrows match the usuadifferl LaTeX arrows 
\usepackage[hidelinks]{hyperref}

\usepackage[numbers,sort&compress]{natbib}
\makeatletter
\renewcommand{\@biblabel}[1]{[#1]\hfill}
\makeatother
\usepackage[mathscr]{euscript}
\usepackage{calrsfs}
\DeclareMathAlphabet{\pazocal}{OMS}{zplm}{m}{n}

\usepackage[shortlabels]{enumitem}

\theoremstyle{plain}
\newtheorem{theorem}[equation]{Theorem}

\newtheorem{corollary}[equation]{Corollary}

\newtheorem{lemma}[equation]{Lemma}
\newtheorem{proposition}[equation]{Proposition}

\theoremstyle{definition}
\newtheorem{definition}[equation]{Definition}
\newtheorem{example}[equation]{Example}
\newtheorem{remark}[equation]{Remark}
\newtheorem{construction}[equation]{Construction}

\numberwithin{equation}{section}
\makeatletter
\binoppenalty=\maxdimen
\relpenalty=\maxdimen
\newlength{\@listleftmargin}
\setlist{leftmargin=\@listleftmargin,itemsep=0pt,topsep=0pt,partopsep=0pt,parsep=\parskip}
\setlist[enumerate]{align=left,labelsep=*,leftmargin=\@listleftmargin,itemsep=0pt,topsep=0pt,partopsep=0pt,parsep=\parskip}
\makeatother

\newcommand{\Triv}{\mathrm{Triv}}
\newcommand{\Ham}{{\mathrm{Ham}}}
\newcommand{\UnSh}{{\mathrm{UnSh}}}
\newcommand{\iu}{\mathrm{i}}

\newcommand{\Tot}{\operatorname{Tot}}

\newcommand{\rmH}{\mathrm{H}}

\newcommand{\rmh}{\mathrm{h}}
\newcommand{\rmD}{\mathrm{D}}

\newcommand{\dR}{\mathrm{dR}}

\newcommand{\dd}{\mathrm{d}}
\newcommand{\End}{\mathrm{End}}

\newcommand{\Hom}{\mathrm{Hom}}
\newcommand{\Iso}{\mathrm{Iso}}

\newcommand{\rank}{\mathrm{rk}}

\newcommand{\cl}{\mathrm{cl}}

\newcommand{\curv}{\mathrm{curv}}

\newcommand{\triv}{{\mathrm{triv}}}
\newcommand{\colim}{{\mathrm{colim}}}

\newcommand{\rmB}{\mathrm{B}}

\newcommand{\scU}{\mathscr{U}}

\newcommand{\scC}{\mathscr{C}}

\newcommand{\scZ}{\mathscr{Z}}

\newcommand{\scR}{\mathscr{R}}

\newcommand{\CG}{\pazocal{G}}
\newcommand{\CI}{\pazocal{I}}

\newcommand{\CT}{\pazocal{T}}

\newcommand{\CF}{\pazocal{F}}

\newcommand{\CE}{\pazocal{E}}

\newcommand{\CA}{\pazocal{A}}

\newcommand{\sfD}{\mathsf{D}}

\newcommand{\sfU}{\mathsf{U}}

\newcommand{\sfc}{\mathsf{c}}

\newcommand{\sfC}{\mathsf{C}}

\newcommand{\hol}{\mathsf{hol}}

\newcommand{\CN}{\mathbb{C}}
\newcommand{\bbD}{{\mathbb{D}}}
\newcommand{\bbH}{{\mathbb{H}}}
\newcommand{\NN}{\mathbb{N}}
\newcommand{\RN}{\mathbb{R}}
\newcommand{\ZN}{\mathbb{Z}}

\newcommand{\bbS}{{\mathbb{S}}}
\newcommand{\bbT}{{\mathbb{T}}}
\newcommand{\bbX}{{\mathbb{X}}}

\newcommand{\frg}{{\mathfrak{g}}}

\newcommand{\PT}{\mathtt{pt}}

\newcommand{\HLBdl}{{\pazocal{HL}\hspace{-0.02cm}\pazocal{B}un}}
\newcommand{\Grb}{{\pazocal{G}rb}}

\newcommand{\HVBdl}{{\pazocal{H}\hspace{-0.025cm}\pazocal{V}\hspace{-0.025cm}\pazocal{B}un}}

\newcommand{\Cat}{{\mathscr{C}\mathrm{at}}}

\newcommand{\Mfd}{{\mathscr{M}\mathrm{fd}}}

\newcommand{\Sh}{{\mathrm{Sh}}}

\newcommand{\ul}[1]{\underline{#1}}

\newcommand{\dslash}{{/\hspace{-0.1cm}/}}

\newcommand{\arisom}{\overset{\cong}{\longrightarrow}}

\newcommand{\cC}{\check{C}}

\newcommand{\<}{\langle}
\renewcommand{\>}{\rangle}

\newenvironment{myitemize}{\begin{itemize}[itemsep=-0.1cm, leftmargin=*, topsep=0cm]}{\end{itemize}}
\newenvironment{myenumerate}{\begin{enumerate}[itemsep=-0.1cm, leftmargin=*, topsep=0cm, label=(\arabic*)]}{\end{enumerate}}

\newcommand{\qen}{\hfill$\triangleleft$}

\newcommand{\qandq}{\quad \text{and} \quad }

\mathtoolsset{showonlyrefs,showmanualtags}

\begin{document}

\begin{flushright}
\small
\textsf{Hamburger Beiträge zur Mathematik Nr.\,886}\\
\textsf{ZMP--HH/21--1}
\end{flushright}

\begin{center}
\LARGE{\textbf{Gerbes in Geometry, Field Theory, and Quantisation}}
\end{center}
\begin{center}
\large Severin Bunk
\end{center}

\begin{abstract}
\noindent
This is a mostly self-contained survey article about bundle gerbes and some of their recent applications in geometry, field theory, and quantisation.
We cover the definition of bundle gerbes with connection and their morphisms, and explain the classification of bundle gerbes with connection in terms of differential cohomology.
We then survey how the surface holonomy of bundle gerbes combines with their transgression line bundles to yield a smooth bordism-type field theory.
Finally, we exhibit the use of bundle gerbes in geometric quantisation of 2-plectic as well as 1- and 2-shifted symplectic forms.
This generalises earlier applications of gerbes to the prequantisation of quasi-symplectic groupoids.
\end{abstract}

\tableofcontents

\section{Introduction}
\label{sec:Intro}

Bundle gerbes on a manifold $M$ are differential geometric representatives for the elements of $\rmH^3(M;\ZN)$, in analogy to how line bundles on $M$ represent the elements of $\rmH^2(M;\ZN)$.
Originally, \emph{gerbes} were introduced as certain sheaves of groupoids by Giraud~\cite{Giraud:Coho_NonAb}, and their popularity in geometry and physics was boosted by Hitchin's notes~\cite{Hitchin:Special_Lagr_Lectures} and Brylinski's book~\cite{Brylinski:Book}.
The concept of \emph{bundle gerbes} goes back to Murray~\cite{Murray:Bundle_gerbs}, who had learned about gerbes from Hitchin~\cite{Murray:Intro_to_gerbes}, and who was looking for a more differential geometric way of describing classes in $\rmH^3(M;\ZN)$.
Since then, the theory of bundle gerbes has been developed further, and various applications of bundle gerbes have been found and studied in mathematics and physics.

The main goal of the present article is to survey the theory of bundle gerbes with connection and some of its applications in a mostly self-contained fashion.
Additionally, we hope that this article may serve as a modern entry point to the area of bundle gerbes.
We assume only basic familiarity with category theory, not going beyond the notions of categories, functors and natural transformation.
The only original contributions of this article are the new presentation of the material, the notion of the curvature of a morphism of gerbes, and the suggestion to use bundle gerbes with connection to treat shifted symplectic quantisation in the world of differential geometry.
We point out that gerbes have also been employed very recently in shifted geometric quantisation in~\cite{Safronov:Shifted_GeoQuan} in the original algebro-geometric context of shifted symplectic structures.
Further, we apologise in advance for any incompleteness of references.
In particular, we do not attempt to present a full literature review in this introduction, but we include numerous references and pointers to further literature throughout the main text.

Let us provide a very basic idea of what a bundle gerbe is:
any hermitean line bundle on a manifold $M$ can be constructed (up to isomorphism) via local $\sfU(1)$-valued transition functions with respect to some open covering $\scU = \{U_a\}_{a \in \Lambda}$ of $M$.
These transition functions are smooth maps $g_{ab} \colon U_{ab} \to \sfU(1)$, where $U_{ab} = U_a \cap U_b$, for $a,b \in \Lambda$, satisfying the cocycle condition
\begin{equation}
	g_{bc} \cdot g_{ab} = g_{ac}
\end{equation}
on each triple overlap $U_{abc}$.
Heuristically speaking, a bundle gerbe is obtained by replacing the transition functions $g_{ab} \colon U_{ab} \to \sfU(1)$ by hermitean line bundles $L_{ab} \to U_{ab}$.
However, since line bundles admit morphisms between them, we cannot simply demand a strict analogue of the cocycle condition of the form ``$L_{bc} \otimes L_{ab} = L_{ac}$'' over triple overlaps.
Instead, we have to specify \emph{how} the two sides of this would-be equation are identified:
on each triple overlap $U_{abc}$ we have to give isomorphisms
\begin{equation}
	\mu_{abc} \colon L_{bc} \otimes L_{ab} \arisom L_{ab}\,,
\end{equation}
and these have to satisfy a version of the \v{C}ech 2-cocycle condition (see Section~\ref{sec:Grbs and twVBuns}).
The idea to replace functions by vector bundles gives great guidance for how to pass from the theory of line bundles to that of bundle gerbes; many rigorous analogies between the two theories can be discovered in this way.

The relevance of (bundle) gerbes includes, but is not limited to, the following results:
gerbes are geometric models for twists of K-theory~\cite{BCMMS, CW--Thom_iso+pfwd_in_twKT}, describe the $B$-field and D-branes in string theory~\cite{Kapustin:D-branes_in_nontriv_B-fields, Gawedzki:branes_in_WZW-models_and_gerbes, Waldorf--Thesis}, and play an important role in topological T-duality~\cite{BEM:T-duality, BN:T-Duality}.
It has been shown that (bundle) gerbes \emph{with connection} even model the third \emph{differential} cohomology of a manifold~\cite{Gajer:Geo_of_Deligne_Coho, Murray-Stevenson:Bgrbs--stable_isomps_and_local_theory}, and that they describe various anomalies in quantum field theory~\cite{CMM:BGrbs_applied_to_QFT, BMS:Smooth_2Grp_Ext_and_GrbSym}.
Bundle gerbes have found additional relevance as sources for twisted Courant algebroids in generalised geometry~\cite{Hitchin:Generalised_CY, Gualtieri:Thesis}, and certain infinitesimal symmetries of gerbes (and bundle gerbes) correspond to the Lie 2-algebra of sections of their associated Courant algebroids~\cite{FRS:L_infty-algs_of_local_obs, Collier:Sym_of_gerbes}.
Further, bundle gerbes with connection on a manifold $M$ correspond to certain line bundles with connection on the free loop space $LM$~\cite{Waldorf--Trangression_II}, and they give rise to smooth bordism-type field theories on $M$ (in the sense of Stolz-Teichner~\cite{ST:SuSy_FTs_and_generalised_coho}) in a functorial manner~\cite{BW:OCTFTs_and_Gerbes}.
Gerbes as well as bundle gerbes have been used in $2$-plectic and shifted geometric quantisation~\cite{LGX--Pre-quasi-sym_quant_via_Grbs, Rogers:Thesis, Bunk--Thesis, Safronov:Shifted_GeoQuan}, where they replace the prequantum line bundle of conventional geometric quantisation.
We survey some of these applications in the main part of this text.
From now on, whenever we use the term `gerbe', we shall mean `bundle gerbe'.

This article is structured as follows:
in Section~\ref{sec:Gerbes}, we first recast the theory of line bundles in a language which will allow us to directly obtain Murray's definition of gerbes with connection through the above process of replacing functions by vector bundles.
In particular, we recall the notion of a simplicial manifold, which we use throughout this article.
Then, we define bundle gerbes with connections, their morphisms, and their 2-morphisms, and survey the tensor product and duals of gerbes, before giving a detailed outline of the classification of gerbes with connection in terms of Deligne cohomology.
Along the way, we introduce the curvature of a morphism of gerbes, show how vector bundles on $M$ act on morphisms of gerbes on $M$, and give an introduction to the Deligne complex as a model for differential cohomology.
We finish this section with the examples of lifting bundle gerbes and cup-product bundle gerbes.

Section~\ref{sec:PT and field theory} is an introduction to the parallel transport of gerbes:
this is defined not just on paths and loops, but also on surfaces in $M$ with and without boundary.
We start with the most well-known case of gerbe holonomy around closed oriented surfaces and introduce the transgression line bundle as a necessary gadget for extending this construction to surfaces with boundary.
We illustrate how this gives rise to a smooth functorial field theory on $M$ in the sense of~\cite{ST:SuSy_FTs_and_generalised_coho, BW:OCTFTs_and_Gerbes}.
We conclude the section with various comments on the inclusion of D-branes into this picture, on the full parallel transport of gerbes with connection, and how the transgression line bundle arises as its holonomy.

Finally, in Section~\ref{sec:Higher Geo Quan} we survey two approaches to geometric quantisation in the presence of higher-degree versions of symplectic forms.
There are two such generalisations in the literature, going by the names of \emph{$n$-plectic forms} and \emph{shifted symplectic forms}.
We demonstrate that gerbes play the role of a higher prequantum line bundle in both cases.
In the $n$-plectic case, we survey Rogers' theory of Poisson Lie $n$-algebras~\cite{Rogers:Thesis} and a recent result by Krepski and Vaughan which relates multiplicative vector fields on a gerbe to its Poisson Lie $2$-algebra.
In the $n$-shifted symplectic case, we first describe derived closed and shifted symplectic forms in differential geometry following Getzler's notes~\cite{Getzler:Slides_on_Stacks}.
Then, we demonstrate how gerbes and their morphisms are perfectly suited to provide higher prequantum line bundles in this setting.
This contains the case of symplectic groupoids, where the notion of curvature of gerbe morphisms introduced here allows us to circumvent the exactness condition on the 3-form part of the shifted symplectic form from~\cite{LGX--Pre-quasi-sym_quant_via_Grbs}.
We finish by relating Waldorf's multiplicative gerbes~\cite{Waldorf:Multiplicative_Gerbes} to the 2-shifted prequantisation of the simplicial manifold $\rmB G$ for any compact, simple, simply connected Lie group $G$.

\vspace{-0.5cm}
\paragraph*{Topics not addressed in this survey}

The literature and relevance of gerbes is too vast to cover every aspect of it in this article.
However, there are several topics which should not go unmentioned entirely (for the same reason, though, the following list is necessarily still not exhaustive):
gerbes and higher gerbes are relevant in index theory; the $n$-form part of the Atiyah-Singer index theorem for families arises as the curvature of an $(n{-}2)$-gerbe~\cite{Lott--Index_Gerbes}.
Further, gerbes and 2-gerbes underlie various smooth models for the string group and control string structures on a manifold (and thus spin structures on its free loop space)~\cite{Waldorf--String_Conns_and_CS-thy, Waldorf:Spin_on_LM_characterising_String_Strs, BMS:Smooth_2Grp_Ext_and_GrbSym, Bunk:Pr_oo-bundles_and_String}.
Certain types of equivariant gerbes can be used to describe geometrically the three-dimensional Kane-Mele invariant of topological phases of matter, see~\cite{Gawedzki--2dFKM, GT--Gauge-theoretic_KM, BS--Top_Ins_and_KM} and references therein.
Finally, all gerbes that appear in this article are abelian (their transition functions are valued in an abelian group, see Section~\ref{sec:Deligne coho and classification}).
There is also a theory of \emph{non-abelian gerbes}---a recent review with further references can be found in~\cite{SW--Non-Ab_Gerbes}---and gerbes can be defined on geometric spaces more general than manifolds; see, for instance,~\cite{Huerta--BGrbs_on_SuperManifolds, Schreiber-DiffCoho_in_CohTopos}.

\vspace{-0.5cm}
\paragraph{Acknowledgements}

The author would like to thank Ezra Getzler and Pavel Safronov for insightful conversations about shifted symplectic forms in derived geometry, and Vicente Cortes, Thomas Mohaupt, and Carlos Shahbazi for various discussions about categorical structures in differential geometry.
Further, the author is grateful to Konrad Waldorf for comments on a first draft of this article.
This work was partially supported by the Deutsche Forschungsgemeinschaft (DFG, German Research Foundation) under Germany's Excellence Strategy---EXC 2121 ``Quantum Universe''---390833306.

\section{Bundle gerbes and their morphisms}
\label{sec:Gerbes}

In this section, we survey the main definitions of gerbes and their morphisms on manifolds, as well as  their tensor product and their duals.
We explain the classification of gerbes with connections in terms of Deligne (i.e.~differential) cohomology and exhibit two classes of gerbes as examples.

\subsection{A simplicial perspective on line bundles}
\label{sec:spl view on LBuns}

As a warm-up, let us recall the construction of hermitean line bundles from transition functions.
If $M$ is a manifold and $\scU = \{U_a\}_{a \in \Lambda}$ is an open covering of $M$, local data for a hermitean line bundle consists of a $\sfU(1)$-valued \v{C}ech 1-cocycle on $\scU$, i.e.~functions $g_{ab} \colon U_{ab} \to \sfU(1)$ such that $g_{bc} g_{ab} = g_{ac}$ on each (non-empty) triple intersection $U_{abc}$.
(Note that we are using the common notation $U_{a_0 \cdots a_n} \coloneqq U_{a_0} \cap \cdots \cap U_{a_n}$ for $a_0, \ldots, a_n \in \Lambda$.)
From these data we can construct a hermitean line bundle $P \to M$ given as
\begin{equation}
	P = \bigg( \coprod_{a \in \Lambda} U_a \times \CN \bigg) / {\sim}\,,
\end{equation}
where the equivalence relation is defined as follows:
we denote the elements in $\coprod_{a \in \Lambda} U_a \times \CN$ by $(a,x,z)$, where $a \in \Lambda$, $x \in U_a$, and $z \in \CN$.
Then, $(a,x,z) \simeq (b,x',z')$ precisely if $x = x'$ and $z' = g_{ab}(x) z$.
The \v{C}ech cocycle relation ensures that this is indeed an equivalence relation.

While open coverings of $M$ are the most common device to describe line bundles in term of transition data, they do not provide the most general choice:
let $\pi \colon Y \to M$ be any surjective submersion.
Consider a smooth map $g \colon Y \times_M Y \to \sfU(1)$, where $Y \times_M Y = \{(y_0,y_1) \in Y^2\, | \, \pi(y_0) = \pi(y_1)\}$ is the manifold of pairs of points in $Y$ which lie in a common fibre over $M$.
If, for every $y_0, y_1, y_2 \in Y$ in a common fibre, we have $g(y_1,y_2) g(y_0,y_1) = g(y_0,y_2)$, we can define
\begin{equation}
\label{eq:PBun from ssub}
	P = (Y \times \CN) /{\sim}\,,
	\qquad
	\text{with}\ (y_0,z_0) \sim \big( y_1, g(y_0,y_1)^{-1}\, z_0 \big) \quad \forall\, (y_0,y_1) \in Y \times_M Y\,.
\end{equation}
This, again, defines a hermitean line bundle on $M$.
(For a precise statement of how this generalises the open-covering picture, see Example~\ref{eg:ssub from opcov}.)

This construction extends to principal $G$ bundles on $M$:
if $G$ is a Lie group and $g \colon Y \times_M Y \to G$ is a smooth map such that $g(y_1,y_2) g(y_0,y_1) = g(y_0,y_2)$ for every $y_0, y_1, y_2 \in Y$ in a common fibre, we obtain a principal $G$-bundle as the quotient
\begin{equation}
	P = (Y \times G)/{\sim}\,,
	\qquad (y_0,h) \sim \big( y_1, g(y_0,y_1)^{-1}\,h \big)\,.
\end{equation}

We can reformulate this construction in the following way, which motivates much of our treatment of bundle gerbes in the later sections.
Given a surjective submersion $\pi \colon Y \to M$ of manifolds, we introduce the following notation; this may seem unnecessarily cumbersome for the description of line bundles at first, but it opens up a very general and powerful perspective on geometric structures, and will appear throughout this article.
For $n \in \NN$, we define the manifolds
\begin{equation}
	\cC Y_{n-1} = Y^{[n]} = Y \times_M \cdots \times_M Y
	= \{ (y_0, \ldots, y_{n-1}) \in Y^n\, | \, \pi(y_i) = \pi(y_j)\ \forall\, i,j = 0, \ldots, n-1 \}\,.
\end{equation}
We define smooth maps
\begin{align}
	&d_i^n \colon \cC Y_n \to \cC Y_{n-1},,
	\quad d^n_i (y_0, \ldots, y_n) = (y_0, \ldots, \widehat{y_i}, \ldots, y_n)\,,
	\\*
	&s^n_i \colon 
	\cC Y_n \to \cC Y_{n+1}\,,
	\quad s^n_i (y_0, \ldots, y_n) = (y_0, \ldots, y_{i-1}, y_i, y_i, y_{i+1}, \ldots, y_n)\,,
\end{align}
where $i \in \{0, \ldots, n-1\}$ and where the hat over an element denotes omission of that element.
In the following we will write $d_i$ and $s_i$ instead of $d^n_i$ and $s^n_i$, respectively, leaving the superscript $n$ as understood from context.
A direct check confirms that the maps $d_i, s_i$ satisfy the so-called \emph{simplicial identities}
\begin{alignat}{3}
\label{eq:spl identities}
	d_i \circ d_j &= d_{j-1} \circ d_i & &\text{if } i < j\,,
	\\
	d_i \circ s_j &= s_{j-1} \circ d_i & &\text{if } i < j\,,
	\\
	d_j \circ s_j &= 1_{\cC Y_n} = d_{j+1} \circ s_j \quad & & 0 \leq j \leq n
	\\
	d_i \circ s_j &= s_j \circ d_{i-1} & &\text{if } i > j+1\,,
	\\
	s_i \circ s_j &= s_{j+1} \circ s_i & &\text{if } i \leq j\,.
\end{alignat}

\begin{definition}
\label{def:spl object}
Let $\scC$ be a category.
A \emph{simplicial object in $\scC$} is a collection $\{X_n\}_{n \in \NN_0}$ of objects $X_n \in \scC$ together with morphisms $d_i \colon X_n \to X_{n-1}$ for $i = 0, \ldots, n$, and $s_i \colon X_n \to X_{n+1}$ for $i = 0, \ldots, n$ (for each $n \in \NN_0$), satisfying the simplicial identities~\eqref{eq:spl identities} (replacing $\cC Y_n$ by $X_n$).
The morphisms $d_i$ and $s_i$ are called the \emph{face maps} and the \emph{degeneracy maps} of $X$, respectively.
We will denote a simplicial object $(\{X_n\}_{n \in \NN_0}, d_i, s_i)$ simply by $X$.
A \emph{morphism $X \to X'$} of simplicial objects in $\scC$ is a collection of morphisms $f = \{f_n \colon X_n \to X'_n\}_{n \in \NN_0}$ with $f_{n-1} \circ d_i = d'_i \circ f_n$ and $f_{n+1} \circ s_n = s'_n \circ f_n$ for all $n$ and $i$.
\end{definition}

\begin{example}
A simplicial object in the category $\Mfd$ of manifolds and smooth maps is called a \emph{simplicial manifold}.
We see from our arguments above that the data $(\cC Y, d_i,s_i)$  obtained from any surjective submersion $\pi \colon Y \to M$ form a simplicial manifold.
We call this simplicial manifold the \emph{\v{C}ech nerve of $\pi \colon Y \to M$}.
\qen
\end{example}

\begin{definition}
\label{def:spl_+ object}
Let $\scC$ be a category.
An \emph{augmented simplicial object in $\scC$} is a simplicial object $X$ in $\scC$ together with an object $X_{-1} \in \scC$ and an additional morphism $d_{-1} = d^0_{-1} \colon X_0 \to X_{-1}$ such that $d_{-1} \circ d_0 = d_{-1} \circ d_1$.
We will denote an augmented simplicial object by $X \to X_{-1}$.
A \emph{morphism $(X \to X_{-1}) \to (X' \to X'_{-1})$} of augmented simplicial objects in $\scC$ is a morphism $f \colon X \to X'$ of simplicial objects together with a morphism $f_{-1} \colon X_{-1} \to X'_{-1}$ such that $f_{-1} \circ d_{-1} = d'_{-1} \circ f_0$.
\end{definition}

\begin{example}
The \v{C}ech nerve of any surjective submersion $\pi \colon Y \to M$ is even an augmented simplicial manifold:
we set $\cC Y_{-1} = M$ and $d_{-1} = \pi$.
\qen
\end{example}

\begin{example}
\label{eg:ssub from opcov}
A particular case of the preceding example arises from open coverings $\scU = \{U_a\}_{a \in \Lambda}$ of a manifold $M$:
setting $\cC \scU_0 \coloneqq \coprod_{a \in \Lambda} U_a$, we obtain a canonical surjective submersion $\pi \colon \cC \scU_0 \to M$.
(A point in $\cC \scU_0$ consists of a pair $(a,x)$ of $a \in \Lambda$ and $x \in U_a$, and $\pi$ sends this pair to $x \in M$.)
The \v{C}ech nerve of $\pi$ agrees with the usual \v{C}ech nerve of the open covering $\scU$, i.e.~we have
\begin{equation}
	\cC \scU_n = \coprod_{a_0, \ldots, a_n \in \Lambda} U_{a_0 \cdots a_n}\,,
\end{equation}
for every $n \in \NN_0$.
\qen
\end{example}

\begin{example}
\label{eg:M//G as spl Mfd}
Let $G$ be a Lie group with neutral element $e \in G$, acting on a manifold $M$ from the right.
We define a simplicial manifold $M \dslash G$ as follows:
we set $(M \dslash G)_n = M \times G^n$ and
\begin{align}
	d_i(x, g_1, \ldots, g_n) &=
	\begin{cases}
		(x \cdot g_1, g_2, \ldots, g_n) & i = 0\,,
		\\
		(x, g_1, \ldots, g_{i-1}, g_i g_{i+1}, g_{i+2}, \ldots, g_n) & 0<i<n\,,
		\\
		(x, g_1, \ldots, g_{n-1}) & i=n\,,
	\end{cases}
	\\
	s_i(x, g_1, \ldots, g_n) &=
	\begin{cases}
		(x, e, g_1, \ldots, g_n) & i=0\,,
		\\
		(x, g_1, \ldots, g_i, e, g_{i+1}, \ldots, g_n) & 0<i<n\,,
		\\
		(x, g_1, \ldots, g_n, e) & i=n\,.
	\end{cases}
\end{align}
The simplicial manifold $M \dslash G$ is the \emph{action Lie $\infty$-groupoid} associated with the $G$-action on $M$.
\qen
\end{example}

\begin{example}
\label{eg:BG as spl Mfd}
In the previous example, if $M = *$ is the one-point manifold carrying the trivial $G$-action, we also write $\rmB G \coloneqq * \dslash G$.
We call this simplicial manifold the \emph{classifying space for (principal) $G$-bundles}.
(This is not the classifying space of $G$-bundles used in algebraic topology; to arrive there, one has to take the geometric realisation of our $\rmB G$.
However, the nomenclature used here receives its justification through Proposition~\ref{st:G-Buns from CY -> BG}.)
\qen
\end{example}

\begin{proposition}
\label{st:G-Buns from CY -> BG}
Let $G$ be a Lie group and $\pi \colon Y \to M$ a surjective submersion.
Transition data for a principal $G$-bundle on $M$ with respect to $\pi$ is the same as a morphism of simplicial manifolds $g \colon \cC Y \to \rmB G$.
\end{proposition}

\begin{proof}
Let $g \colon \cC Y \to \rmB G$ be a morphism of simplicial manifolds.
This consists of a collection of smooth maps $g_n \colon \cC Y_n \to \rmB G_n = G^n$.
We can thus write the map $g_n$ as an $(n+1)$-tuple $g_n = (g_{n,1}, \ldots, g_{n,n})$ of maps $g_{n,i} \colon \cC Y \to G$.
Observe that the map $g_0 \colon Y \to \rmB G_0 = *$ is trivial.
The compatibility of $g$ with the face and degeneracy maps has the following consequences:
for $g_1 = (g_{1,1})$, we obtain the normalisation condition $g_1(y,y) = g_1 \circ s_0(y) = s_0 \circ g_0(y) = e$ for all $y \in Y$.
For $g_2 = (g_{2,1}, g_{2,2})$, we have $g_{2,1} = d_2 \circ g_2 = g_1 \circ d_2$, or equivalently $g_{2,1}(y_0,y_1,y_2) = g_1(y_0,y_1)$, for all $(y_0,y_1,y_2) \in Y^{[3]} = \cC Y_2$.
Analogously, using $d_0$ instead of $d_2$ we obtain $g_{2,2}(y_0,y_1,y_2) = g_1(y_1,y_2)$.
That is, the map $g_2 \colon Y^{[3]} \to G^2$ is completely determined by $g_1$:
we have $g_2(y_0,y_1,y_2) = (g_1(y_0,y_1),\, g_1(y_1,y_2))$ for all $(y_0,y_1,y_2) \in Y^{[3]} = \cC Y_2$.
Finally, using $d_1$, we obtain that $g_{2,0} \cdot g_{2,1} = d_1 \circ g_2 = g_1 \circ d_1$, and with our previous findings, this yields
\begin{equation}
\label{eq:1-Ccocycle for ssub}
	d_0^*g_1 \cdot d_2^*g_1 = d_1^*g_1\,,
\end{equation}
which is precisely the cocycle condition~\eqref{eq:PBun from ssub} we need to build a principal bundle from $g_1$.

For $g_3 = (g_{3,1}, g_{3,2}, g_{3,3})$, we can use that $g_{3,1} = d_2 \circ d_3 \circ g_3 = g_1 \circ d_2 \circ d_3$, and similarly for the other components, so that $g_3$ is completely determined by $g_1$ as well.
Inductively, we derive that this holds true for $g_n$, for any $n \geq 2$, and that the remaining compatibilities with the face and degeneracy maps follow readily from the cocycle condition~\eqref{eq:1-Ccocycle for ssub}.
\end{proof}

We have thus seen that we can encode principal $G$-bundles in morphisms of simplicial manifolds, for any Lie group $G$.
Let us now include connections.
For simplicity, we will restrict ourselves to $\sfU(1)$-bundles, i.e.~to the case of $G = \sfU(1)$.
In this case, connection forms are represented locally by 1-forms valued in $\iu\RN$.
To formulate connections in sufficient generality for our purposes, we first need the following dual notion of a simplicial object:

\begin{definition}
\label{def:cosimplicial object}
Let $\scC$ be a category.
A \emph{cosimplicial object in $\scC$} is a collection $\{X_n\}_{n \in \NN_0}$ of objects $X_n \in \scC$ together with morphisms $\partial_i \colon X_{n-1} \to X_n$ for $i = 0, \ldots, n$, and $\sigma_i \colon X_{n+1} \to X_n$ for $i = 0, \ldots, n$ (for each $n \in \NN_0$), satisfying the \emph{cosimplicial identities}
\begin{alignat}{3}
\label{eq:cospl identities}
	\partial_j \circ \partial_i &= \partial_i \circ \partial_{j-1} & &\text{if } i < j\,,
	\\
	\sigma_j \circ \partial_i &= \partial_i \circ \sigma_{j-1} & &\text{if } i < j\,,
	\\
	\sigma_j \circ \partial_j &= 1_{X_n} = \sigma_j \circ \partial_{j+1} \quad & & 0 \leq j \leq n
	\\
	\sigma_j \circ \partial_i &= \partial_{i-1} \circ \sigma_j & &\text{if } i > j+1\,,
	\\
	\sigma_j \circ \sigma_i &= \sigma_i \circ \sigma_{j+1} & &\text{if } i \leq j\,.
\end{alignat}
The morphisms $\partial_i$ and $\sigma_i$ are called the \emph{coface maps} and the \emph{codegeneracy maps} of $X$, respectively.
Often, we will denote a cosimplicial object $(\{X_n\}_{n \in \NN_0}, \partial_i, \sigma_i)$ simply by $X$.
A \emph{morphism $X \to X'$} of cosimplicial objects in $\scC$ is a collection of morphisms $f = \{f_n \colon X_n \to X'_n\}_{n \in \NN_0}$ with $f_n \circ \partial_i = \partial'_i \circ f_{n-1}$ and $f_n \circ \sigma_n = \sigma'_n \circ f_{n+1}$ for all $n$ and $i$.
\end{definition}

\begin{construction}
\label{cons:AltFace coch complex for csp VSp}
Let $X$ be a simplicial manifold.
For any $k \in \NN_0$, we obtain a family of real vector spaces $\{\Omega^k(X_n)\}_{n \in \NN_0}$.
The face and degeneracy maps of $X$ induce pullback maps $\partial_i = d_i^* \colon \Omega^k(X_{n-1}) \to \Omega^k(X_n)$ and $\sigma_i = s_i^* \colon \Omega^k(X_{n+1}) \to \Omega^k(X_n)$, respectively, which satisfy the cosimplicial identities~\eqref{eq:cospl identities}.
More concretely, $(\Omega^k(X), d_i^*, s_i^*)$ is a cosimplicial object in the category of real vector spaces and linear maps.
For each $n \in \NN_0$, we define the linear maps
\begin{equation}
	\delta \colon \Omega^k(X_n) \to \Omega^k(X_{n+1})\,,
	\qquad
	\delta(\omega) = \sum_{i=0}^{n+1} (-1)^i\, \partial_i(\omega)\,,
\end{equation}
which make $(\Omega^k(X), \delta)$ into a cochain complex of $\RN$-vector spaces.
Note that the same construction extends to augmented simplicial manifolds $X \to M$, giving a complex $(\Omega^k(X \to M),\delta)$ with $\Omega^k(M)$ in degree $-1$ and $\delta \colon \Omega^k(M) \to \Omega^k(X)$ given by $d_{-1}^*$.
\qen
\end{construction}

\begin{lemma}
\label{st:Cech coho on ssub}
\emph{\cite[Sec.~8]{Murray:Bundle_gerbs}}
Let $Y \to M$ be a surjective submersion.
For any $k \in \NN_0$, the complex $(\Omega^k(\cC Y), \delta)$ has trivial cohomology in all non-zero degrees.
The complex $(\Omega^k(\cC Y \to M), \delta)$ has trivial cohomology in all degrees.
\end{lemma}

We also obtain that for any finite-dimensional $\RN$-vector space $V$, the complex $(\Omega^k(\cC Y \to M;V),\delta)$ of $V$-valued $k$-forms has trivial cohomology, for any $k \in \NN_0$.

\begin{proposition}
Let $Y \to M$ be a surjective submersion, and let $g \colon \cC Y \to \rmB \sfU(1)$ be a morphism of simplicial manifolds.
Let $\mu_{\sfU(1)} \in \Omega^1(\sfU(1);\iu \RN)$ denote the Maurer-Cartan form on $\sfU(1)$.
\begin{myenumerate}
\item We have $\delta (g_1^*\mu_{\sfU(1)}) = 0$.

\item The data of a connection form $A \in \Omega^1(Y;\iu\RN)$ for the principal $\sfU(1)$-bundle defined by $g$ is the same as a coboundary $A \in \Omega^1(\cC Y_0,\iu\RN)$ which trivialises the cocycle $g_1^*\mu_{\sfU(1)} \in \Omega^2(\cC Y_1;\iu\RN)$.
\end{myenumerate}
\end{proposition}

\begin{remark}
\label{rmk:conn as flatness witness}
If we view $g_1^*\mu_{\sfU(1)}$ as the \emph{curvature of the transition functions}, then a connection on the bundle defined by $g \colon \cC Y \to \rmB \sfU(1)$ is precisely a way of witnessing that the curvature $g_1^*\mu_{\sfU(1)}$ is trivial \emph{up to a specified homotopy} in $(\Omega^1(\cC Y; \iu\RN), \delta)$.
\qen
\end{remark}

\subsection{Gerbes and twisted vector bundles}
\label{sec:Grbs and twVBuns}

In this section we survey the definition of bundle gerbes and their morphisms.
Bundle gerbes were introduced in~\cite{Murray:Bundle_gerbs}, and this theory was further developed in particular in~\cite{Murray-Stevenson:Bgrbs--stable_isomps_and_local_theory,BCMMS,Waldorf--More_morphisms,Waldorf--More_morphisms,BSS--HGeoQuan,Bunk--Thesis}.
A short-hand approach to the same theory from a higher sheaf theoretic perspective has been developed in~\cite{NS--Equivariance_in_higher_geometry}.

In Section~\ref{sec:spl view on LBuns} we used $\sfU(1)$-valued functions of manifolds to construct hermitean line bundles (or $\sfU(1)$-bundles).
That is, we used objects with little structure---$\sfU(1)$-valued functions first of all form a set%
\footnote{The collection of $\sfU(1)$-valued functions on $M$ also has an abelian group structure, which induces the tensor product of line bundles.
We come to the analogue of this additional algebraic structure for gerbes in Section~\ref{sec:operations on BGrbs}.}%
---to obtain new objects with more structure:
the collection of line bundles has two layers of structure consisting of the objects (i.e.~the line bundles) and their morphisms.
That is, line bundles on a manifold form a \emph{category} rather than a set.

In order to define bundle gerbes, we iterate this idea:
we now aim to use line bundles as transition data for new geometric objects.
With this goal in mind, we imitate the constructions of Section~\ref{sec:spl view on LBuns}.
Thus, let $\pi \colon Y \to M$ be a surjective submersion.
The new `transition functions' now consist of a (hermitean) line bundle $L \to \cC Y_1$, replacing the function $g_1 \colon \cC Y_1 \to \sfU(1)$.
The key to constructing $\sfU(1)$-bundles from transition functions $g_1 \colon \cC Y_1 \to \sfU(1)$ in Section~\ref{sec:spl view on LBuns} was the cocycle condition~\eqref{eq:1-Ccocycle for ssub}.
Replacing the product in $\sfU(1)$ by the tensor product of line bundles, here this amounts to choosing an isomorphism
\begin{equation}
\label{eq:BGrb mu}
	\mu \colon d_0^*L \otimes d_2^*L \arisom d_1^*L
\end{equation}
of line bundles over $\cC Y_2$.
(Since line bundles admit morphisms between them, we can---and have to---specify \emph{how} the two sides of the cocycle condition on $L$ are identified, rather than simply saying that they are equal.)
For consistency, $\mu$ has to satisfy the following coherence condition:
the diagram
\begin{equation}
\label{eq:assoc for BGrb mu}
\begin{tikzcd}[column sep=1.25cm, row sep=1cm]
	p_{23}^*L \otimes p_{12}^*L \otimes p_{01}^*L \ar[r, "1 \otimes p_{012}^*\mu"] \ar[d, "p_{123}^*\mu \otimes 1"']
	& p_{23}^*L \otimes p_{02}^*L \ar[d, "p_{023}^*\mu"]
	\\
	p_{13}^*L \otimes p_{01}^*L \ar[r, "p_{013}^*\mu"']
	& p_{03}^*L
\end{tikzcd}
\end{equation}
of line bundles over $\cC Y_3$ commutes.
Here, $p_{i_0 \cdots i_k} \colon Y^{[n+1]} \to Y^{[k+1]}$, $(y_0, \ldots, y_n) \mapsto (y_{i_0}, \ldots, y_{i_k})$ are the smooth projection maps.
Note that these can be written as compositions of the face maps $d_i$ in non-unique ways; for instance, $p_{01} \colon Y^{[4]} \to Y^{[2]}$ can be written as $p_{01} = d_2 \circ d_2 = d_2 \circ d_3$.

\begin{definition}
\label{def:BGrb CG}
Let $M$ be a manifold.
A \emph{bundle gerbe} on $M$ is a tuple $\CG = (\pi \colon Y \to M, L, \mu)$ consisting of a surjective submersion $\pi$, a line bundle $L \to \cC Y_1$, and an isomorphism of line bundles $\mu$ as in~\eqref{eq:BGrb mu} which satisfies the coherence condition~\eqref{eq:assoc for BGrb mu}.
\end{definition}

\begin{definition}
\label{def:BGrb w conn}
Let $M$ be a manifold carrying a bundle gerbe $\CG = (\pi \colon Y \to M, L, \mu)$.
A \emph{connection on $\CG$} is a pair of a unitary connection on the line bundle $L$ and a 2-form $B \in \Omega^2(\cC Y_0; \iu\RN)$ such that $\curv(L) = - \delta B$ in $\Omega^2(\cC Y_1; \iu\RN)$, where $\curv(L)$ is the curvature of the connection on $L$.
The 2-form $B$ is called the \emph{curving} of $\CG$.
By Lemma~\ref{st:Cech coho on ssub}, there exists a unique closed 3-form $H \in \Omega^3(M;\iu\RN)$ with $\pi^*H = \dd B$; we set $\curv(\CG) = H$ and call this the \emph{curvature 3-form of $\CG$}.
\end{definition}

As a consequence of Lemma~\ref{st:Cech coho on ssub}, every bundle gerbe admits a connection~\cite{Murray:Bundle_gerbs}.

\begin{remark}
Compare the condition $\curv(L) = - \delta B$ to Remark~\ref{rmk:conn as flatness witness}:
here, $\curv(L)$ is the curvature of the transition data (replacing $g_1^*\mu_{\sfU(1)}$), and $-B$ is a degree-zero element in $\Omega^2(\cC Y; \iu\RN)$ which trivialises $\curv(L)$ in $(\Omega^2(\cC Y; \iu\RN), \delta)$.
\qen
\end{remark}

We now define morphisms of bundle gerbes.
These are again modelled on morphisms of line bundles on $M$, which are defined in terms of transition functions $g,g' \colon \cC Y \to \rmB \sfU(1)$:
morphisms of line bundles correspond to functions $\psi \colon \cC Y_0 \to \CN$ satisfying
\begin{equation}
	d_0^*\psi \cdot g_1 = g'_1 \cdot d_1^*\psi
	\quad
	\text{over } \cC Y_1\,.
\end{equation}
To obtain morphisms of gerbes, we replace functions by vector bundles and identities by isomorphisms.
However, in general two bundle gerbes $\CG_0, \CG_1$ will be defined over different surjective submersions, so that we can only compare the gerbes after choosing a common refinement of the surjective submersions.

\begin{definition}
\label{def:1mp of BGrbs}
Let $\CG_i = (\pi_i \colon Y_i \to M, L_i, \mu_i)$, for $i = 0,1$, be two bundle gerbes on a manifold $M$.
A \emph{morphism of bundle gerbes} $\CE \colon \CG_0 \to \CG_1$ is a tuple $\CE = (\zeta \colon Z \to Y_1 \times_M Y_2, E, \alpha)$, consisting of the following data:
$\zeta$ is a surjective submersion onto $Y_1 \times_M Y_2 = \{(y,y') \in Y_1 \times Y_2\, | \, \pi_1(y) = \pi_2(y')\}$, and $E \to Z$ is a hermitean vector bundle.
The composition $Z \to Y_1 \times_M Y_2 \to M$ is a surjective submersion to $M$ with \v{C}ech nerve $\cC Z$.
Then,
\begin{equation}
	\alpha \colon d_0^*E \otimes L_0 \longrightarrow L_1 \otimes d_1^*E
\end{equation}
is an isomorphism of hermitean vector bundles over $\cC Z_1 = Z \times_M Z$.
This has to be compatible with $\mu_0$ and $\mu_1$ in the sense that the following diagram of vector bundles over $\cC Z_2$ commutes:
\begin{equation}
\begin{tikzcd}[column sep=1.5cm, row sep=1cm]
	p_2^*E \otimes p_{12}^*L_0 \otimes p_{01}^*L_0 \ar[r, "p_{12}^* \alpha \otimes 1"] \ar[d, "1 \otimes \mu_0"']
	& p_{12}^*L_1 \otimes p_1^*E \otimes p_{01}^*L_0 \ar[r, "1 \otimes p_{01}^* \alpha"]
	& p_{12}^*L_1 \otimes p_{01}^*L_1 \otimes p_3^*E \ar[d, "\mu_1 \otimes 1"]
	\\
	p_2^*E \otimes p_{02}^*L_0 \ar[rr, "p_{02}^*\alpha"']
	& & p_{02}^*L_1 \otimes p_0^*E
\end{tikzcd}
\end{equation}
\end{definition}

Note that we have omitted the pullbacks of $L_i$ from $(\cC Y_i)_1 = Y_i \times_M Y_i$ to $\cC Z_1$ along the maps $\cC Z_1 \to (\cC Y_i)_1$ induced by $\zeta$ in order to avoid overly cluttered notation.

\begin{definition}
\label{def:1mp of BGrbs w conn}
If $\CG_0, \CG_1$ additionally carry connections, then a \emph{morphism of gerbes with connection} is a morphism $\CE = (\zeta \colon Z \to Y_0 \times_M Y_1, E, \alpha) \colon \CG_0 \to \CG_1$ of the underlying bundle gerbes where the hermitean vector bundle $E$ additionally carries a unitary connection and $\alpha$ is connection-preserving.
We call a morphism $\CE \colon \CG_0 \to \CG_1$ \emph{parallel} (equivalently, it satisfies the \emph{fake curvature condition}) if $\curv(E) = B_0 - B_1$ in $\Omega^2(\cC Z_0; \iu \RN)$.
\end{definition}

Recall that in Section~\ref{sec:spl view on LBuns} we constructed geometric objects with more structure (line bundles, which form a category) out of geometric objects with less structure ($\sfU(1)$-valued functions, which form a set).
The same is true here:
morphisms of bundle gerbes are built from vector bundles, which are again geometric objects that admit morphisms between them.
This leads to:

\begin{definition}
\label{def:2mp of BGrbs}
Let $\CG_0, \CG_1$ be bundle gerbes on $M$, and suppose $\CE = (\pi \colon Z \to Y_0 \times_M Y_1, E, \alpha)$ and $\CE' = (\pi' \colon Z' \to Y_0 \times_M Y_1, E', \alpha')$ are two morphisms $\CG_0 \to \CG_1$.
A \emph{2-morphism} $\CE \to \CE'$ is an equivalence class of tuples $(w \colon W \to Z \times_M Z', \psi)$, where $w$ is a surjective submersion, and where $\psi \colon E \to E'$ is a morphism of hermitean vector bundles over $W$ (where we have omitted the pullback maps for $E$ and $E'$), making the diagram
\begin{equation}
\begin{tikzcd}[column sep=1.25cm, row sep=1cm]
	d_0^*E \otimes L_0 \ar[r, "\alpha"] \ar[d, "d_0^*\psi \otimes 1"']
	& L_1 \otimes d_1^*E \ar[d, "1 \otimes d_1^*\psi"]
	\\
	d_0^*E' \otimes L_0 \ar[r, "\alpha'"']
	& L_1 \otimes d_1^*E'
\end{tikzcd}
\end{equation}
of hermitean vector bundles over $\cC W_1 = W \times_M W$ commute.
Two tuples $(w \colon W \to Z \times_M Z', \psi)$, $(\tilde{w} \colon \tilde{W} \to Z \times_M Z', \tilde{\psi})$ are equivalent if there exists a surjective submersion $v \colon V \to W \times_M \tilde{W}$ such that the pullbacks of $\psi$ and $\tilde{\psi}$ to $V$ agree.
\end{definition}

We will usually denote a 2-morphism $[w \colon W \to Z \times_M Z', \psi]$ just as $\psi \colon \CE \to \CE'$.
The definition of 2-morphisms of bundle gerbes carries over verbatim to bundles gerbes with connection.
In that situation, $\CE$ and $\CE'$ carry connections, and we call $\psi \colon \CE \to \CE'$ \emph{parallel} whenever the underlying morphism $\psi \colon E \to E'$ of vector bundles is connection-preserving.

We thus have \emph{three} layers of structure for bundle gerbes, given by bundle gerbes, their morphisms, and their 2-morphisms.
The main structural result is the following theorem.
We refer the reader to Appendix~\ref{app:2Cats} for a very brief survey of 2-categories.

\begin{theorem}
\label{st:Grb(M) is 2Cat}
Let $M$ and $N$ be manifolds, and let $f \colon N \to M$ be a smooth map.
\begin{myenumerate}
\item The collection of bundle gerbes on $M$ forms a 2-category $\Grb(M)$.

\item The collection of bundle gerbes with connection on $M$ forms a 2-category $\Grb^\nabla(M)$.

\item In both $\Grb(M)$ and $\Grb^\nabla(M)$, a morphism $\CE \colon \CG_0 \to \CG_1$ is invertible (see Definition~\ref{def:invertible mp in 2Cat}) if and only if its underlying hermitean vector bundle $E$ has rank one.

\item The smooth map $f$ induces pullback 2-functors
\begin{equation}
	f^* \colon \Grb(M) \to \Grb(N)
	\qandq
	f^* \colon \Grb^\nabla(M) \to \Grb^\nabla(N)\,,
\end{equation}
satisfying $\curv(f^*\CG) = f^*\curv(\CG)$, for any $\CG \in \Grb^\nabla(M)$.
\end{myenumerate}
\end{theorem}

The notion of morphisms of gerbes goes back to~\cite{Murray-Stevenson:Bgrbs--stable_isomps_and_local_theory}, where only line bundles were used.
This was extended to general vector bundles in~\cite{BCMMS,Waldorf--More_morphisms}.

\begin{example}
The composition of $\CE \colon \CG_0 \to \CG_1$ and $\CE' \colon \CG_1 \to \CG_2$ is given by
\begin{equation}
	\CE' \circ \CE
	= \big( Z' \times_{Y_1} Z \to Y_0 \times_M Y_2,\, E' \otimes E,\, \alpha' \otimes \alpha \big)\,,
\end{equation}
where we have again omitted pullbacks.
For any gerbe $\CG$ (possibly with connection), the identity morphism $1_\CG \colon \CG \to \CG$ is given by $1_\CG = (1_{Y \times_M Y}, L, p_{013}^*\mu^{-1} \circ p_{023}^*\mu)$.
\qen
\end{example}

\begin{remark}
One can show that any morphism $\CE \colon \CG_0 \to \CG_1$ is 2-isomorphic to a morphism $\CE'$ whose surjective submersion is the identity on $Y_0 \times_M Y_1$ (see~\cite[Thm.~2.4.1]{Waldorf--Thesis} and~\cite[Thm.~A.19]{Bunk--Thesis}).
Further, every 2-morphism $\psi = [W \to Z \times_M Z', \psi]$ has a unique representative whose surjective submersion is the identity on $Z \times_M Z'$~\cite[Prop.~A.16]{Bunk--Thesis}.
However, including the general choices of surjective submersions ensures that we have a functorial composition of morphisms in $\Grb^\nabla(M)$ rather than a weaker notion of composition.
\qen
\end{remark}

\subsection{Operations on gerbes and their morphisms}
\label{sec:operations on BGrbs}

We saw in Theorem~\ref{st:Grb(M) is 2Cat} that one can pull back gerbes, their morphisms and their 2-morphisms along smooth maps of manifolds.
In this section we survey further operations on gerbes and their morphisms.
These operations are again motivated by the operations one can perform on the category of hermitean line bundles.
We present all constructions for gerbes with connection; the corresponding versions for gerbes without connections are obtained simply by forgetting all connections.

\begin{definition}
\label{def:tensor in Grb(M)}
Let $\CG_0, \CG_1, \CG'_0, \CG'_1 \in \Grb^\nabla(M)$ be bundle gerbes with connection on $M$, let $\CE, \CF \colon \CG_0 \to \CG_1$ and $\CE', \CF' \colon \CG'_0 \to \CG'_1$ be morphisms in $\Grb(M)$, and let $\psi \colon \CE \to \CF$ and $\varphi \colon \CE' \to \CF'$ be 2-morphisms.
\begin{myenumerate}
\item The \emph{tensor product of gerbes} $\CG_0 = (\pi_0 \colon Y_0 \to M, L_0, \mu_0, B_0)$ and $\CG'_0 = (\pi'_0 \colon Y'_0 \to M, L'_0, \mu'_0, B'_0)$ is the bundle gerbe with connection
\begin{equation}
	\CG_0 \otimes \CG'_0 = (Y_0 \times_M Y'_0 \to M,\, L_0 \otimes L'_0,\, \mu_0 \otimes \mu'_0,\, B_0 + B'_0)\,,
\end{equation}
where we are omitting pullbacks along the projections $Y_0 \leftarrow Y_0 \times_M Y'_0 \to Y'_0$.

\item The \emph{tensor product of morphisms} $\CE = (Z \to Y_0 \times_M Y_1, E, \alpha)$ and $\CE' = (Z' \to Y'_0 \times_M Y'_1, E', \alpha')$ reads as
\begin{equation}
	\CE \otimes \CE' = \big( Z \times_M Z' \to (Y_0 \times_M Y'_0) \times_M (Y_1 \times_M Y'_1),\, E \otimes E',\, \alpha \otimes \alpha' \big)\,,
\end{equation}
where we have omitted pullbacks and canonical isomorphisms which rearrange tensor products of vector bundles, such as $E \otimes L_0 \otimes E' \otimes L'_0 \cong L_0 \otimes L'_0 \otimes E \otimes E'$.

\item The \emph{tensor product of 2-morphisms} $\psi = [W \to Z \times_M X, \psi]$ and $\varphi = [W' \to Z' \times_M X', \varphi]$ is given by
\begin{equation}
	\psi \otimes \varphi = [W \times_M W' \to (Z \times_M Z') \times_M (X \times_M X'),\, \psi \otimes \varphi]\,.
\end{equation}
\end{myenumerate}
\end{definition}

These tensor products are compatible with compositions of morphisms and 2-morphisms~\cite{Waldorf--Thesis}.
For $A \in \Omega^1(M; \iu\RN)$, denote the trivial line bundle with connection given by $A$ as $I_A = (M \times \CN, A)$.
The trivial line bundle with connection, $I_0 = (M \times \CN, 0)$ is the monoidal unit (the neutral element) for the tensor product of line bundles.

\begin{example}
\label{eg:CI_rho}
Let $\rho \in \Omega^2(M; \iu\RN)$.
The \emph{trivial bundle gerbe on $M$ with connection given by $\rho$} is
\begin{equation}
	\CI_\rho = (1_M \colon M \to M, L = I_0, m, \rho)\,,
\end{equation}
where $m \colon I_0 \otimes I_0 \to I_0$ is the morphism given by $(x,z) \otimes (x,z') \mapsto (x, z \cdot z')$ for $x \in M$, $z,z' \in \CN$.
The \emph{trivial bundle gerbe with connection} is $\CI_0$; this is the monoidal unit for the tensor product of bundle gerbes with connection~\cite{Waldorf--More_morphisms}.
\qen
\end{example}

\begin{theorem}
\label{st:Grb(M) is sym mon 2Cat}
\emph{\cite{Waldorf--More_morphisms, Waldorf--Thesis}}
For any manifold $M$, the tensor product of gerbes, their morphisms and their 2-morphisms turns $(\Grb^\nabla(M), \otimes, \CI_0)$ into a symmetric monoidal 2-category.
\end{theorem}

Given a vector bundle $E \to M$, we denote its dual vector bundle by $E^\vee$.
Given a morphism $\psi \colon E \to F$ of vector bundles, we denote its dual morphism by $\psi^\vee \colon F^\vee \to E^\vee$.

\begin{definition}
\label{def:duals in Grb(M)}
Let $\CG_0, \CG_1 \in \Grb^\nabla(M)$ be bundle gerbes with connection on $M$, let $\CE, \CE' \colon \CG_0 \to \CG_1$, and let $\psi \colon \CE \to \CE'$ be a 2-morphism.
\begin{myenumerate}
\item The \emph{dual gerbe} of $\CG_0 = (\pi_0 \colon Y_0 \to M, L_0, \mu_0, B_0)$ is $\CG_0^\vee = (\pi_0 \colon Y_0 \to M, L^\vee_0, \mu^{-\vee}_0, -B_0)$.

\item The \emph{dual morphism} of $\CE = (\zeta \colon Z \to Y_0 \times_M Y_1, E, \alpha)$ is $\CE^\vee = (sw \circ \zeta \colon Z \to Y_1 \times_M Y_0, E^\vee, \alpha^{-\vee})$, where $sw \colon Y_0 \times_M Y_1 \to Y_1 \times_M Y_0$ is the swap of factors.
This defines a morphism $\CE^\vee \colon \CG_1^\vee \to \CG_0^\vee$.

\item The \emph{dual 2-morphism} $\psi^\vee \colon \CE'{}^\vee \to \CE^\vee$ is obtained by sending the bundle morphism $E \to E'$ underlying the 2-morphism $\psi$ to its dual bundle morphism $E'{}^\vee \to E^\vee$.
\end{myenumerate}
\end{definition}

\begin{remark}
There are several variations on the notion of duals of morphisms and 2-morphisms of gerbes, some of which only reverse the direction of morphisms but not of 2-morphisms.
Certain definitions of duals may be more adapted to particular problems, see e.g.~\cite{Waldorf--Thesis,Bunk--Thesis}.
\qen
\end{remark}

\begin{proposition}
\emph{\cite{Waldorf--Thesis}}
For any $\CG \in \Grb^\nabla(M)$, there are canonical parallel isomorphisms $\CG^\vee \otimes \CG \to \CI_0$ and $\CI_0 \to \CG \otimes \CG^\vee$, establishing $\CG^\vee$ as the categorical dual of $\CG$.
\end{proposition}

The following observation has important consequences for the existence of morphisms between gerbes (see Corollary~\ref{st:torsion problem}):

\begin{proposition}
\label{st:det trick}
If $\CE \colon \CG_0 \to \CG_1$ is a morphism in $\Grb(M)$ (or $\Grb^\nabla(M))$ with underlying vector bundle $E$, then the determinant bundle $\det(E)$ induces an isomorphism $\det(\CE) \colon \CG_0^{\otimes \rank(E)} \arisom \CG_1^{\otimes \rank(E)}$.
\end{proposition}

\begin{proof}
One can directly see that $\det(E)$ induces a morphism $\det(\CE) \colon \CG_0^{\otimes \rank(E)} \arisom \CG_1^{\otimes \rank(E)}$; since $\rank(\det(E)) = 1$, this is an isomorphism by Theorem~\ref{st:Grb(M) is 2Cat}(3).
\end{proof}

\begin{construction}
\label{cons:tensoring mps by VBuns}
Let $\CE \colon \CG_0 \to \CG_1$ be a morphism of bundle gerbes with connection on $M$, given as the tuple $\CE = (Z \to Y_0 \times_M Y_1, E, \alpha)$.
Further, let $F$ be a hermitean vector bundle with connection on $M$.
We can form a new morphism (omitting the pullback of $F$ along $Y_0 \times_M Y_1 \to M$)
\begin{equation}
\label{eq:tensoring mps by VBuns}
	F \otimes \CE \colon \CG_0 \to \CG_1\,,
	\qquad
	F \otimes \CE = (Z \to Y_0 \times_M Y_1, F \otimes E, 1_F \otimes \alpha)\,.
\end{equation}
If $\varphi \colon F \to F'$ is a morphism of hermitean vector bundles, and $\psi = [W \to Z \times_M Z', \psi] \colon \CE \to \CE'$ is a 2-morphism of bundle gerbes (with some $\CE' \colon \CG_0 \to \CG_1$), then we set
\begin{equation}
	\varphi \otimes \psi = [W \to Z \times_M Z', \varphi \otimes \psi] \colon F \otimes \CE \longrightarrow F' \otimes \CE'\,.
\end{equation}
This construction is compatible with the tensor product of vector bundles and composition; in other words, there is a (module) action
\begin{equation}
\label{eq:HVB action on Grb Homs}
	\otimes \colon \HVBdl^\nabla(M) \times \Hom_{\Grb^\nabla(M)}(\CG_0, \CG_1) \longrightarrow \Hom_{\Grb^\nabla(M)}(\CG_0, \CG_1)
\end{equation}
of the symmetric monoidal category $\HVBdl^\nabla(M)$ of hermitean vector bundles with connection on $M$ on the category of morphisms from $\CG_0$ to $\CG_1$, for any pair of gerbes with connection on $M$.
\qen
\end{construction}

\begin{remark}
The module action~\eqref{eq:HVB action on Grb Homs} is a categorified analogue of the fact that the set of morphisms $L_0 \to L_1$ between any line bundles on $M$ is a module over $C^\infty(M)$.
\qen
\end{remark}

\begin{construction}
\label{cons:HVBun enrichment}
There is a dual operation to the action~\eqref{eq:HVB action on Grb Homs}:
let $\CE, \CE' \colon \CG_0 \to \CG_1$ be morphisms between bundle gerbes on $M$.
After choosing a common refinement of the underlying surjective submersion, we may assume that $\CE$ and $\CE'$ are defined over the same surjective submersion $Z \to Y_0 \times_M Y_1$.
We then consider the $\Hom$ bundle $\ul{\Hom}(E,E')$, which comes with the isomorphism $\beta$ defined as the composition
\begin{equation}
\begin{tikzcd}[column sep=1.2cm, row sep=1cm]
	d_0^* \ul{\Hom}(E,E') \ar[r, "(-) \otimes 1_{L_0}", "\cong"'] \ar[d, dashed, "\beta"']
	& d_0^* \ul{\Hom}(E,E') \otimes \ul{\Hom}(L_0,L_0) \ar[r, "\cong"]
	& \ul{\Hom}(L_0 \otimes d_0^*E, L_0 \otimes d_0E') \ar[d, "{\ul{\Hom}(\alpha^{-1}, \alpha')}"]
	\\
	d_1^* \ul{\Hom}(E,E')
	& \ul{\Hom}(L_1, L_1) \otimes d_1^* \ul{\Hom}(E,E') \ar[l, "f \otimes \psi \mapsto f\psi", "\cong"']
	&  \ul{\Hom}(d_1^*E \otimes L_1, d_1^*E' \otimes L_1) \ar[l, "\cong"]
\end{tikzcd}
\end{equation}
This satisfies the cocycle identity, and thus the pair $(\ul{\Hom}(E,E'), \beta)$ induces a unique (up to canonical isomorphism) hermitean vector bundle with connection on $M$, which we denote by $\ul{\Hom}(\CE,\CE')$.
We refer to~\cite[Sec.~4.5]{BSS--HGeoQuan} and~\cite[Sec.~3.2]{Bunk--Thesis} for full details for this construction.
\qen
\end{construction}

\begin{proposition}
\label{st:2isos as sections}
\emph{\cite[Prop.~4.37]{BSS--HGeoQuan}, \cite[Cor.~3.67]{Bunk--Thesis}}
Let $\CE, \CE' \colon \CG_0 \to \CG_1$ be morphisms of gerbes with connection.
There is a canonical bijection (which is even an isomorphism of $C^\infty(M)$-modules)
\begin{equation}
\label{eq:2isos as sections}
	2\Hom_{\Grb^\nabla(M)}(\CE, \CE') \cong \Gamma \big( M; \ul{\Hom}(\CE,\CE') \big)
\end{equation}
between 2-morphisms $\CE \to \CE'$ and sections of the bundle $\ul{\Hom}(\CE,\CE')$.
Furthermore, under this isomorphism, parallel 2-morphisms correspond to parallel sections, and unitary isomorphisms correspond to unit-length sections.
\end{proposition}

Let $\CE \colon \CG_0 \to \CG_1$ be a morphism of gerbes with connection.
Since $\alpha$ preserves connections (see Definition~\ref{def:1mp of BGrbs w conn}) the 2-form $\curv(E) + (B_1 - B_0)$ descends to a bundle-valued 1-form $\curv(\CE) \in \Omega^2(M;\ul{\End}(\CE))$; here, we have set $\ul{\End}(\CE) = \ul{\Hom}(\CE,\CE)$.

\begin{definition}
We call $\curv(\CE) \in \Omega^2(M;\ul{\End}(\CE))$ the \emph{curvature} of the morphism $\CE \colon \CG_0 \to \CG_1$.
\end{definition}

Note that if $\rank(E) = 1$, then $\ul{\End}(\CE)$ is trivial; it is a hermitean line bundle on $M$ with a canonical unit-length section given by the identity 2-isomorphism $1_\CE \colon \CE \to \CE$ under the isomorphism~\eqref{eq:2isos as sections}.

\begin{proposition}
\label{st:par up to connection shift}
If $\CE \colon \CG_0 \to \CG_1$ is an isomorphism, then there exists a unique 2-form $\rho \in \omega^2(M; \iu\RN)$ such that $\CE \colon \CG_0 \otimes \CI_\rho \to \CG_1$ is parallel.
\end{proposition}

\begin{proof}
By Theorem~\ref{st:Grb(M) is 2Cat}(3), we have $\rank(E) = 1$.
The fact that $\alpha$ is parallel implies that $\delta(\curv(E) - (B_1 - B_0)) = 0$.
The claim then follows from Lemma~\ref{st:Cech coho on ssub}:
there exists (a unique) $\rho \in \Omega^2(M; \iu\RN)$ such that the pullback of $\rho$ to $Z$ equals $-(\curv(E) + (B_1 - B_0))$.
(Note that we precisely have $\rho = - \curv(\CE)$ and that $\CG_1 \otimes \CI_\rho$ can be identified with the gerbe $(Y_1 \to M, L_1, \mu_1, B_1 + \rho)$.)
\end{proof}

\begin{remark}
Using a (pseudo-)Riemannian metric on $M$ and the structure on $\End(\CE)$ induced from $\ul{\End}(E)$ one can now define a Yang-Mills functional on the connections on $\CE$ compatible with those on $\CG_0$ and $\CG_1$, using the curvature $\curv(\CE)$.
This defines \emph{twisted Yang-Mills theory}, which, to our knowledge, has so far not been investigated.
Physically, it describes Yang-Mills theory on the world-volume of D-branes in string theory in the presence of non-trivial Kalb-Ramond charge.
Note that $\curv(\CE)$ (or its trace) does \emph{not} have integer cycles in general.
\qen
\end{remark}

Using the above techniques one can prove the following results:
for $\eta \in \Omega^2(M; \iu\RN)$ with integer cycles, let $\HLBdl^\nabla_\eta(M)$ be the groupoid of hermitean line bundles on $M$ with connection of curvature $\eta$ and their unitary parallel isomorphisms.
For bundle gerbes $\CG_0, \CG_1 \in \Grb^\nabla(M)$, let \smash{$\Iso^{par}_{\Grb^\nabla(M)}(\CG_0, \CG_1)$} denote the groupoid of parallel gerbe isomorphisms and their unitary parallel 2-isomorphisms.
Finally, recall the trivial gerbe with connection $\CI_\rho$ from Example~\ref{eg:CI_rho}.

\begin{proposition}
\emph{\cite{Waldorf--More_morphisms,Bunk--Thesis}}
For any $\rho, \rho' \in \Omega^2(M; \iu\RN)$, there are equivalences of categories
\begin{align}
	\Hom_{\Grb^\nabla(M)}(\CI_\rho, \CI_{\rho'}) &\simeq \HVBdl^\nabla(M)\,,
	\\*
	\Iso^{par}_{\Grb^\nabla(M)}(\CI_\rho, \CI_{\rho'}) &\simeq \HLBdl^\nabla_{\rho'-\rho}(M)\,.
\end{align}
\end{proposition}

\subsection{Deligne cohomology and the classification of gerbes with connection}
\label{sec:Deligne coho and classification}

Consider a manifold $M$, and suppose we are given a \emph{differentiably good} open covering $\scU = \{U_a\}_{a \in \Lambda}$ of $M$.
Here, differentiably good means that each finite intersection of the patches $U_a$ is either empty or diffeomorphic to $\RN^m$ with $m = \dim(M)$.
Recall from Example~\ref{eg:ssub from opcov} that $\scU$ induces a surjective submersion $\cC \scU_0 \to M$ whose \v{C}ech nerve agrees with the usual \v{C}ech nerve of the open covering $\scU$.
In order to give a gerbe with connection defined over this surjective submersion, we have to specify a line bundle $L \to \cC \scU_1$.
However, since $\cC \scU_1 = \coprod_{a,b \in \Lambda} U_{ab}$ is a disjoint union of manifolds diffeomorphic to $\RN^m$, we may assume without  loss of generality that $L = I_A$, i.e.~that $L$ is the trivial line bundle with connection given by some $A \in \Omega^1(\cC\scU_1; \iu\RN)$.
We further have to specify a parallel bundle isomorphism $\mu \colon d_0^*L \otimes d_2^*L \to d_1^*L$; since $L = I_A$, such an isomorphism is equivalent to a smooth function $g \colon \cC \scU_2 \to \sfU(1)$ satisfying
\begin{equation}
	d_0^*A + d_2^*A = d_1^*A + g^*\mu_{\sfU(1)}\,,
	\quad
	\text{or equivalently}
	\quad
	\delta A = g^*\mu_{\sfU(1)}
\end{equation}
in $(\Omega^1(\cC \scU_1; \iu\RN), \delta)$.
The coherence condition~\eqref{eq:assoc for BGrb mu} for $\mu$ then translates to the condition that
\begin{equation}
	d_0^*g \cdot d_2^*g = d_1^*g\,,
	\quad
	\text{or equivalently}
	\quad
	\delta g = 0
\end{equation}
in $(C^\infty(\cC \scU; \sfU(1)), \delta)$, the \v{C}ech complex associated to $\cC \scU$ and the sheaf of $\sfU(1)$-valued functions.
(This \v{C}ech complex is obtained in a way analogous to Construction~\ref{cons:AltFace coch complex for csp VSp}.)
This is precisely the condition that $g$ be a $\sfU(1)$-valued \v{C}ech 2-cocycle on $M$ with respect to the open covering $\scU$.
Finally, we have to give a curving for our bundle gerbe, which consists of a 2-form $B \in \Omega^2(\cC \scU_0; \iu\RN)$ such that
\begin{equation}
	\dd A = - \delta B
\end{equation}
in $(\Omega^2(\cC \scU; \iu\RN), \delta)$, where we have used Definition~\ref{def:BGrb w conn} and Construction~\ref{cons:AltFace coch complex for csp VSp}.

A bundle gerbe as described here is often called a \emph{local bundle gerbe} or a \emph{Hitchin-Chatterjee gerbe}, after \cite{Chatterjee:Gerbs}.
Observe that every surjective submersion $Y \to M$ admits local sections defined over some open covering $\scU$ of $M$, and after possibly choosing a refinement we may assume that $\scU$ is differentiably good.
The local sections assemble to give a smooth map $s \colon \scU \to Y$, commuting with the maps to $M$.
This further induces a morphism $\cC s \colon (\cC \scU \to M) \longrightarrow (\cC Y \to M)$ of augmented simplicial manifolds (see Definition~\ref{def:spl_+ object}) with $s_{-1} = 1_M$.
This is the key step in proving:

\begin{lemma}
Every bundle gerbe with connection on $M$ is isomorphic (by a parallel isomorphism) to a local gerbe $(g,A,B)$ with connection, defined over some differentiably good open covering $\scU$ of $M$.
\end{lemma}

(This statement can be refined to allow one to fix $\scU$ arbitrarily and to also include morphisms and 2-morphisms is~\cite[Thm.~2.101]{Bunk--Thesis}.)
If we hope to classify gerbes with connection on $M$ up to parallel isomorphism, we may thus restrict ourselves to local gerbes.
For details, we refer to~\cite[Sec.~1.2]{Waldorf--Thesis}.
In a similar vein, one can show that parallel isomorphisms $(g,A,B) \to (g',A',B')$ of local gerbes over the same open covering $\scU$ are equivalently given by a function $h \colon \cC \scU_1 \to \sfU(1)$ together with a 1-form $C \in \Omega^1(\cC \scU_0; \iu\RN)$ satisfying
\begin{equation}
	g\, g'{}^{-1} = \delta h\,,
	\quad
	A - A' = h^*\mu_{\sfU(1)} - \delta C
	\qandq
	B - B' = \dd C\,.
\end{equation}
Finally, a parallel 2-isomorphism $(h,C) \to (h',C')$ corresponds to a function $\psi \colon \cC \scU_0 \to \sfU(1)$ with
\begin{equation}
	h\, h'{}^{-1} = - \delta \psi
	\qandq
	C - C' = \psi^*\mu_{\sfU(1)}\,.
\end{equation}
The local data for gerbes and their isomorphisms are conveniently described via the following enhancement of Construction~\ref{cons:AltFace coch complex for csp VSp}:

\begin{construction}
\label{cons:totalisation}
Suppose $C$ is a presheaf of cochain complexes (of abelian groups) on the category $\Mfd$ of smooth manifolds and smooth maps.
That is, $C$ is an assignment $M \to C(M)$ of a cochain complex to each manifold, and a morphism of complexes $f^* \colon C(M) \to C(M')$ to each smooth map $f \colon M' \to M$ such that $(f_1 \circ f_0)^* = f_0^* \circ f_1^*$ for each pair of composable smooth maps, and such that $1_M^*$ is the identity.
We further assume that $C^l(M) = 0$ for each $l > 0$ and each manifold $M$.

Suppose that $(X = \{X_k\}_{k \in \NN_0}, d_i, s_i)$ is a simplicial manifold.
We obtain a family of cochain complexes $\{C(X_k)\}_{k \in \NN_0}$, and like in Construction~\ref{cons:AltFace coch complex for csp VSp} the face and degeneracy maps of $X$ induce morphisms of cochain complexes $\partial_i = d_i^* \colon C(X_{k-1}) \to C(X_k)$ and $\sigma_i = s_i^* \colon C(X_{k+1}) \to C(X_k)$, satisfying the cosimplicial identities~\eqref{eq:cospl identities}.
We can now apply Construction~\ref{cons:AltFace coch complex for csp VSp} to each level of these chain complexes:
thereby we obtain, for each $l \in \ZN$, a cochain complex $(C^l(X), \delta)$ (which is trivial for each $l > 0$).
In fact, since this construction is compatible with the differential $d$ on $C$, we even obtain a \emph{double cochain complex} $(C(X), d, \delta)$.
(Note that this means, in particular, that $d \circ \delta = \delta \circ d$.)
Finally, we take the total complex $(\Tot(C(X)), \rmD)$ of $(C(X), d, \delta)$:
this is the (ordinary) cochain complex of abelian groups with
\begin{equation}
	\Tot \big( C(X) \big)^l = \bigoplus_{p+q = l} C^p(X_q)
	\qandq
	\rmD_{|C^p(X_q)} = d + (-1)^p \delta\,.
\end{equation}
This construction allows us to define not only \v{C}ech hypercohomology, but also derived closed forms (see Section~\ref{sec:shifted symplectic}).
\qen
\end{construction}

\begin{definition}
Let $M$ be a manifold and $n \in \NN_0$.
The \emph{degree-$n$ Deligne complex} of $M$ is the cochain complex of abelian groups
\begin{equation}
	\hat{\rmB}^n_\nabla \sfU(1)(M) = \big(
\begin{tikzcd}[column sep=0.75cm]
	0 \ar[r]
	& C^\infty \big( M,\sfU(1) \big) \ar[r, "\dd \log"]
	& \Omega^1(M; \iu\RN) \ar[r, "\dd"]
	& \cdots \ar[r, "\dd"]
	& \Omega^n(M; \iu\RN) \ar[r]
	& 0
\end{tikzcd}
	\big)
\end{equation}
where $\dd \log(g) = g^*\mu_{\sfU(1)}$, and where $C^\infty \big( M,\sfU(1) \big)$ lies in degree $-n$.
The \emph{degree-$n$ \v{C}ech-Deligne cohomology groups of $M$} are given by
\begin{equation}
	\check{\bbH}^k(M;\hat{\rmB}^n_\nabla\sfU(1))
	= \underset{\scU}{\colim}\ \rmH^k \big( \Tot( \hat{\rmB}^n_\nabla \sfU(1) (\cC\scU)), \rmD \big)\,,
\end{equation}
where $\scU$ runs over all open coverings of $M$.
The \emph{$n$-th differential cohomology group of $M$} is
\begin{equation}
	\hat{\rmH}^n(M;\ZN) = \check{\bbH}^0(M;\hat{\rmB}^{n-1}_\nabla\sfU(1))\,.
\end{equation}
\end{definition}

\begin{remark}
There is a way of defining Deligne cohomology without using \v{C}ech nerves; this uses the concept of \emph{hypercohomology} of complexes of sheaves of abelian groups; for details and background, see e.g.~\cite{Brylinski:Book,Voisin:Hodge_and_Cplx_AG_I}.
\qen
\end{remark}

Note that there are slightly different conventions in some of the literature, amounting to a degree-shift of one in the definition of $\hat{\rmH}^n(M;\ZN)$.
The following general statement allows us to compute \v{C}ech-Deligne cohomology groups:

\begin{proposition}
\label{st:hypercoho via single covering}
\emph{\cite[Thm.~2.8.1]{EZT:Sheaf_Coho_to_AlgDR_Thm}}
Let $C$ be a complex of \emph{sheaves}%
\footnote{That is, each $C^k$ is a sheaf with respect to open coverings of any manifold.}
of abelian groups on $\Mfd$ such that $C^k = 0$ for each $k < 0$.
Let $\scU$ be an open covering of a manifold $M$ such that each $C^k$ induces an acyclic sheaf on each finite intersection $U_{a_0 \cdots a_m}$ of patches of $\scU$ (i.e.~the complex $C(U_{a_0 \cdots a_m})$ has non-trivial cohomology at most in degree zero).
Then there is a canonical isomorphism between $\rmH^n(\Tot(C(\cC \scU)), \rmD)$ and the hypercohomology of $C$ on $M$.
\end{proposition}

\begin{corollary}
Let $\scU$ be a differentiably good open covering of a manifold $M$.
Then, there is a canonical isomorphism
\begin{equation}
	\rmH^k \big( \Tot( \hat{\rmB}^n_\nabla \sfU(1) (\cC\scU)), \rmD \big)
	\cong \check{\bbH}^k(M;\hat{\rmB}^n_\nabla\sfU(1))\,.
\end{equation}
\end{corollary}

\begin{proof}
Each of the level sheaves of $\hat{\rmB}^n_\nabla \sfU(1)[-n]$ (where $[-n]$ denotes the degree shift by $-n$) is acyclic on each $U_{a_0 \cdots a_m}$ by the Poincaré Lemma and the long exact sequence in cohomology associated to the short exact sequence $\ZN \to \RN \to \sfU(1)$.
Thus, the claim follows from Proposition~\ref{st:hypercoho via single covering} and the fact that on paracompact spaces \v{C}ech hypercohomology agrees with hypercohomology~\cite[Thm.~1.3.13]{Brylinski:Book}.
\end{proof}

From our investigation of local bundle gerbes, we see that a local gerbe with connection defined with respect to a differentiably good open covering $\scU$ of $M$ is the same as a 0-cocycle
\begin{equation}
	(g,A,B) \in Z^0 \big( \Tot( \hat{\rmB}^2_\nabla \sfU(1) (\cC\scU)), \rmD \big)\,,
\end{equation}
and two such local gerbes are isomorphic (via a parallel isomorphism) precisely if their cocycles differ by a coboundary.
One can check that restricting cocycles along refinements of differentiably good open coverings does not change their class in \v{C}ech-Deligne cohomology; thus we arrive a the following crucial classification theorem for gerbes with connection on a manifold $M$:

\begin{theorem}
\label{st:classification of Grbs}
\emph{\cite{Murray-Stevenson:Bgrbs--stable_isomps_and_local_theory}}
Let $M$ be a manifold.
There are isomorphisms of abelian groups
\begin{align}
	\bbD \colon \big( \Grb^\nabla(M), \otimes \big) / \text{\{par.~iso\}}
	&\arisom \check{\bbH}^2(M; \hat{\rmB}^2_\nabla \sfU(1) \big)
	\cong \hat{\rmH}^3(M; \ZN)\,,
	\\*[0.1cm]
	\rmD \colon \big( \Grb(M), \otimes \big) / \text{\{iso\}}
	&\arisom \rmH^2 \big( M; \sfU(1) \big)
	\cong \rmH^3(M; \ZN)\,.
\end{align}
\end{theorem}

The second isomorphism is obtained in complete analogy with the first, but replacing the complex $\hat{\rmB}^n_\nabla \sfU(1)$ by the complex $\hat{\rmB}^n \sfU(1) = \sfU(1)[n]$; for a manifold $M$, the complex $(\hat{\rmB}^n\sfU(1))(M)$ has $C^\infty(M, \sfU(1))$ in degree $-n$ and is trivial in all other degrees.
An analogous theorem for gerbes in the sense of Giraud has been proven in~\cite{Gajer:Geo_of_Deligne_Coho}.

\begin{definition}
For $\CG \in \Grb(M)$, we call the class $\rmD(\CG)$ the \emph{Dixmier-Douady class} of $\CG$.
For $\CG \in \Grb^\nabla(M)$ a gerbe with connection, we call the class $\bbD(\CG)$ associated to $\CG$ in $\hat{\rmH}^3(M;\ZN)$ the \emph{Deligne class} of $\CG$.
We also write $\rmD(\CG)$ for the Dixmier-Douady class of $\CG$ with its connection forgotten.
\end{definition}

We summarise the key technical properties of \v{C}ech-Deligne cohomology for the context of gerbes:

\begin{proposition}
\emph{\cite{Brylinski:Book}}
For any manifold $M$ and $n \geq 1$ there are the following exact sequences of abelian groups:
\begin{equation}
\label{eq:Sequences in DiffCoho}
\begin{tikzcd}[row sep=0.1cm, column sep={3cm,between origins}]
	0 \ar[r]
	& \Omega^{n-1}_{\cl, \ZN}(M; \iu\RN) \ar[r]
	& \Omega^{n-1}(M; \iu\RN) \ar[r, "\triv"]
	& \hat{\rmH}^n(M;\ZN) \ar[r, "\sfc"]
	& \rmH^n(M;\ZN) \ar[r]
	& 0
	\\
	& 0 \ar[r]
	& \rmH^{n-1}(M; \sfU(1)_\delta) \ar[r]
	& \hat{\rmH}^n(M; \ZN) \ar[r, "\curv"]
	& \Omega^n_{\cl, \ZN}(M; \iu\RN) \ar[r]
	& 0
\end{tikzcd}
\end{equation}
\end{proposition}

Here, $\sfU(1)_\delta$ is the sheaf of locally constant $\sfU(1)$-valued functions.
Let us look at the sequences~\eqref{eq:Sequences in DiffCoho} for gerbes, i.e.~for $n=3$.
In the first sequence, the map $\triv$ sends a 2-form $\rho$ to the Deligne class $\bbD(\CI_\rho)$ of the trivial gerbe with connection $\rho$.
The second map $\sfc$, also called the \emph{characteristic class} or the \emph{Chern class}, takes the Deligne class $\bbD(G)$ of a gerbe $\CG$ with connection and sends it to the Dixmier-Douady class $\rmD(\CG)$ of the underlying gerbe without its connection.
The first map in the second sequence takes a \v{C}ech 2-cocycle $g \in Z^0(\Tot(\hat{\rmB}^2\sfU(1)_\delta(\cC \scU)),\rmD)$ of locally constant $\sfU(1)$-valued functions and sends it to the Deligne class of the local gerbe $(g,0,0)$.
The map $\curv$ sends $\bbD(\CG)$ to the 3-form $\curv(\CG)$.

\begin{proposition}
\label{st:diffcoho hexagon}
For any $n \in \NN_0$, the exact sequences~\eqref{eq:Sequences in DiffCoho} fit into the \emph{differential cohomology hexagon}~\cite{SiSu:Axiomatic_DiffCoho}, which is the commutative diagram of abelian groups
\begin{equation}
\label{eq:diffcoho hexagon}
\begin{tikzcd}[row sep={1.5cm,between origins}, column sep={2.25cm,between origins}]
	0 \ar[dr]
	& 
	& 
	& 
	& 0
	\\
	{}
	& \frac{\Omega^n(M; \iu\RN)}{\Omega^n_{\cl,\ZN}(M; \iu\RN)} \ar[rr, "\dd"] \ar[dr, "\triv"]
	& 
	& \Omega^{n+1}_{\cl,\ZN}(M; \iu\RN) \ar[ur] \ar[dr, "\widetilde{\dR}"]
	& 
	\\
	\rmH^n(M;\iu \RN_\delta) \ar[ur, "\dR"] \ar[dr, "\exp_*"']
	& 
	& \hat{\rmH}^{n+1}(M;\ZN) \ar[ur, "\curv"] \ar[dr, "\sfc"]
	& 
	& \rmH^{n+1}(M; \iu\RN_\delta)
	\\
	{}
	& \rmH^n(M; \sfU(1)_\delta) \ar[ur] \ar[rr, "\beta"]
	& 
	& \rmH^{n+1}(M; \ZN) \ar[ur, "(\iu \cdot (-))_*"'] \ar[dr]
	& 
	\\
	0 \ar[ur]
	& 
	& 
	& 
	& 0
\end{tikzcd}
\end{equation}
The morphisms $\dR$ and $\widetilde{\dR}$ arise from the de Rham isomorphism $\rmH^\bullet(M; \RN_\delta) \cong \rmH^\bullet_\dR(M;\RN)$ from sheaf to de Rham cohomology, and where $\beta$ is the usual Bockstein homomorphism.
\end{proposition}

The top-left-to-bottom-right short exact sequence implies:

\begin{corollary}
Suppose $\CG_0, \CG_1 \in \Grb^\nabla(M)$ are two bundle gerbes with connection such that $\rmD(\CG_0) = \rmD(\CG_1)$; that is, $\CG_0$ and $\CG_1$ are isomorphic as gerbes \emph{without} connection (see Theorem~\ref{st:classification of Grbs}).
Then, there exists an isomorphism $\CE \colon \CG_0 \to \CG_1$ of gerbes with connection.
In particular, there is a \emph{parallel} isomorphism $\CE \colon \CG_0 \otimes \CI_{-\curv(\CE)} \longrightarrow \CG_1$ (by Proposition~\ref{st:par up to connection shift}).
\end{corollary}

\begin{corollary}
\label{st:torsion problem}
If there exists any morphism $\CE \colon \CG_0 \to \CG_1$ between two gerbes in $\Grb^\nabla(M)$, then the difference $\rmD(\CG_1) - \rmD(\CG_0)$ is a torsion element in $\rmH^3(M;\ZN)$.
\end{corollary}

\begin{proof}
This is a direct consequence of Lemma~\ref{st:det trick} and Theorem~\ref{st:classification of Grbs}.
\end{proof}

\begin{definition}
\label{def:trivialisations of gerbes}
Let $\CG \in \Grb^\nabla(M)$.
A \emph{trivialisation of $\CG$} is a parallel isomorphism $\CT \colon \CI_\rho \to \CG$ for some $\rho \in \Omega^2(M; \iu\RN)$.
We call a bundle gerbe $\CG \in \Grb^\nabla(M)$ \emph{trivialisable} if it admits a trivialisation.
If $\CT' \colon \CI_{\rho'} \to \CG$ is a second trivialisation, an \emph{isomorphism of trivialisations of $\CG$} is a unitary parallel 2-isomorphism $\psi \colon \CT \to \CT'$ in $\Grb^\nabla(M)$.
We let $\Triv(\CG)$ denote the groupoid of trivialisations of $\CG$.
\end{definition}

\begin{proposition}
\label{st:existence of trivs}
Let $\CG \in \Grb^\nabla(M)$ be a gerbe with connection on $M$.
\begin{myenumerate}
\item $\CG$ is trivialisable (i.e.~$\Triv(\CG) \neq \emptyset$) precisely if $\rmD(\CG) = 0$ in $\rmH^3(M;\ZN)$.

\item If $\CT \colon \CI_\rho \to \CG$ and $\CT' \colon \CI_{\rho'} \to \CG$ are trivialisations, then there exists a canonical isomorphism of trivialisations
\begin{equation}
	\ul{\Hom}(\CT,\CT') \otimes \CT \arisom \CT'
\end{equation}
(where we have used Construction~\ref{cons:HVBun enrichment}), and we have
\begin{equation}
	\curv \big( \ul{\Hom}(\CT,\CT') \big) = \rho' - \rho \quad \in \Omega^2_{\cl,\ZN}(M; \iu\RN)\,.
\end{equation}
\end{myenumerate}
\end{proposition}

\begin{proof}
Claim~(1) follows readily from Proposition~\ref{st:diffcoho hexagon}.
The isomorphism in claim~(2) is either obtained directly from the construction of $\ul{\Hom}(\CT,\CT')$, or, more abstractly, from the fact that $\ul{\Hom}$ provides an internal hom functor for the tensor product of gerbe morphisms~\cite{BSS--HGeoQuan,Bunk--Thesis}.
The curvature identity is again a direct consequence of Proposition~\ref{st:diffcoho hexagon}, or can be seen explicitly from the construction of $\ul{\Hom}(\CT, \CT')$ and the fact that $\CT$ and $\CT'$ are parallel morphisms of gerbes.
\end{proof}

\begin{remark}
Because of Corollary~\ref{st:torsion problem}, attempts to allow for infinite-dimensional bundles have been made (see e.g.~\cite[Sec.~7]{BCMMS} and~\cite{CW--Thom_iso+pfwd_in_twKT}), but Kuiper's Theorem (the contractibility of the unitary group of an infinite-dimensional separable Hilbert space) prevents this from yielding good categories of morphisms~\cite[Prop.~4.91]{Bunk--Thesis}.
One can consider Hilbert bundles of reduced structure group instead, but this leads to conflicts with tensor products.
Presumably, circumventing the torsion constraint will involve passing from bundles to sheaves of modules over $C^\infty(M)$, but we leave this to future work.
\qen
\end{remark}

\begin{remark}
Morphisms of gerbes are also called bundle gerbe (bi-)modules~\cite{Waldorf--Thesis}, or \emph{twisted vector bundles}~\cite{BCMMS}.
Bundle gerbe morphisms are related to twisted K-theory~\cite{BCMMS, CW--Thom_iso+pfwd_in_twKT}, at least when the bundle gerbe represents a torsion class in $\rmH^3(M;\ZN)$; otherwise, one has to use ($\ZN_2$-graded) $\infty$-dimensional Hilbert bundles (with reduced structure groups) in place of hermitean vector bundles as morphisms of bundle gerbes, which works fine for the purposes of twisted K-theory~\cite{BCMMS, CW--Thom_iso+pfwd_in_twKT, AS:Twisted_K-theory, Karoubi:Twisted_K-theory}.
\qen
\end{remark}

We conclude this section with a couple of remarks on Deligne complexes, homotopy theory, and higher gerbes.
These remarks are not relevant for the remaining sections of this article, but we hope that they provide an entry point to the extensive works on higher gerbes by Schreiber and collaborators (see, for instance,~\cite{Schreiber-DiffCoho_in_CohTopos,FSS:Higher_stacky_perspective,FSS:Cech_cocycles_for_diff_chars}).

\begin{definition}
\label{def:n-gerbe}
An $n$-gerbe on a manifold $M$ is a cocycle
\begin{equation}
	(g, A_1, \ldots, A_{n+1}) \in Z^0 \big( \Tot( \hat{\rmB}^n_\nabla \sfU(1) (\cC \scU)), \rmD \big)\,.
\end{equation}
\end{definition}

It would be tedious to define morphisms and higher morphisms of $n$-gerbes by hand.
However, there exists a general construction, called the \emph{Dold-Kan correspondence}, which turns a complex of abelian groups into a simplicial set (i.e.~a simplicial object in the category of sets and maps in the sense of Definition~\ref{def:spl object}).
The simplicial set we obtain under this correspondence from the complex $\tau_{\geq 0}(\Tot(\hat{\rmB}^n_\nabla\sfU(1)(\cC \scU), \rmD)$ can be viewed as an $(n{+}1)$-groupoid of $n$-gerbes with connection on $M$ and their parallel morphisms.
($\tau_{\geq 0}$ denotes the truncation to non-negative degrees.)
This is made precise by the fact that a particular type of simplicial sets, called \emph{Kan complexes}, provide a model for $\infty$-groupoids.
Note that it is essential to not evaluate $\hat{\rmB}^n_\nabla\sfU(1)$ on the manifold $M$ itself, but instead on the \v{C}ech nerve of a \emph{differentiably good} open covering $\scU$ of $M$.
There is a formal reason for this, stemming from homotopy theory (one needs to perform a cofibrant replacement of $M$ in a local projective model category of simplicial presheaves).
Abstract homotopy theory and $\infty$-categories are the basis for the theory of higher gerbes developed in~\cite{Lurie:HTT,Schreiber-DiffCoho_in_CohTopos}; good references on homotopy theory and higher categories include~\cite{Quillen:Homotopical_Algebra,Dwyer-Spalinksi,Hovey:Model_Categories,Lurie:HTT,Riehl:Cat_HoThy,Cisinski:Higher_Categories}.

\begin{remark}
For 2-gerbes, a hands-on definition in the spirit of Section~\ref{sec:spl view on LBuns} and~\ref{sec:Grbs and twVBuns} is still feasible; one follows the same principle as in those sections, using gerbes with connection as local transition functions to define 2-gerbes with connection.
For background and examples, we refer the reader to~\cite{Stevenson:Bundle_2-Gerbes,CJMSW:BGrbs_for_CS_and_WZW,Waldorf:Multiplicative_Gerbes}.
\qen
\end{remark}

\subsection{Lifting bundle gerbes and cup product bundle gerbes}
\label{sec:Examples of gerbes}

In this section, we exhibit two examples of bundle gerbes which appear frequently.
For various further examples of gerbes, we refer the interested reader to~\cite{Hitchin:Special_Lagr_Lectures,Chatterjee:Gerbs,Murray-Stevenson:Bgrbs--stable_isomps_and_local_theory,GR:Basic_Gerbe,Meinrenken:Basic_Gerbe,Waldorf:Multiplicative_Gerbes,Waldorf--Trangression_II, BS:Fluxes_Gerbes_2HSpaces} and references therein.

\vspace{-0.4cm}
\paragraph{Lifting bundle gerbes}

Let $\sfU(1) \to G \to H$ be a central extension of Lie groups.
In particular, $G \to H$ is a principal $\sfU(1)$-bundle.
The fact that it is also a group extension can be rephrased as follows:
there exists an isomorphism
\begin{equation}
	\mu^G \colon d_0^*G \otimes d_2^*G \arisom d_1^*G
\end{equation}
of $\sfU(1)$-bundles over $H^2$, which satisfies a verbatim analogue of equation~\eqref{eq:assoc for BGrb mu} over $H^3$.
Observe that these data look suspiciously like a bundle gerbe defined using the simplicial manifold $\rmB H$ in place of the \v{C}ech nerve of a surjective submersion $Y \to M$.

\begin{remark}
If we were able to establish the simplicial manifold $\rmB H$ as the the \v{C}ech nerve of a certain morphism $\pi \colon * = \rmB H_0 \to N$, we would see that a $\sfU(1)$-extension of $H$ is the same as a gerbe on $N$ with respect to $\pi$.
This can be made precise using a theory of principal $\infty$-bundles and their classifying spaces~\cite{Lurie:HTT,NSS:Pr-oo-buns_theory,Bunk:Pr_oo-bundles_and_String}; in this theory, a $\sfU(1)$-extension of $H$ is the same as a gerbe on the classifying space of $H$-bundles, and it is a shadow of this fact that we are observing here.
However, introducing principal $\infty$-bundles would lead us too far afield here.
\qen
\end{remark}

Suppose, $P \to M$ is a principal $H$-bundle on a manifold $M$.
This gives rise to a \v{C}ech nerve, which is the augmented simplicial manifold $\cC P \to M$.
By the principality of the $H$-action on $P$, this comes with a canonical smooth map $h \colon \cC P \to \rmB H$.
In particular, $h$ is fully determined by its level-one component $h_1 \colon \cC P_1 \to H$, satisfying $d_0^*h_1 \cdot d_2^*h_1 = d_1^*h_1$ (cf.~Section~\ref{sec:spl view on LBuns}).
Let $L = h_1^*G$ be line bundle associated to the pullback of the $\sfU(1)$-bundle $G \to H$ along the map $h_1$, and let $\mu = h_2^*\mu^G$.
Then, $\CG = (Y \to M, L, \mu)$ defines a bundle gerbe on $M$, called the \emph{lifting bundle gerbe of $P$}.

\begin{theorem}
\emph{\cite{Murray:Bundle_gerbs}}
The principal $H$-bundle $P \to M$ is a reduction of a principal $G$-bundle if and only if the gerbe $\CG$ is trivialisable.
\end{theorem}

If the principal $H$-bundle $P \to M$ carries a connection, one can extend the construction of the lifting gerbe to also include a connection~\cite{Gomi:Conns_on_lifting_gerbes}.

\vspace{-0.4cm}
\paragraph{Cup product bundle gerbes}

We now consider a manifold $M$, a hermitean line bundle $L \to M$ with connection, and a smooth map $f \colon M \to \bbS^1$.
We can understand $L$ as representing a class in $\hat{\rmH}^2(M;\ZN)$ and $f$ as representing a class in $\hat{\rmH}^1(M;\ZN)$ (since $\bbS^1 \cong \sfU(1)$ as manifolds).
From these data one can construct a gerbe with connection on $M$ which represents the cup-product $[L] \cup [f] \in \hat{\rmH}^3(M;\ZN)$~\cite{Johnson:Thesis}.
This construction proceeds as follows:
let $\RN \to \ZN$, $r \mapsto \exp(2\pi \iu\, r)$ denote the canonical $\ZN$-principal bundle over $\bbS^1$.
Let $p \colon Y = f^*\RN \to M$ denote the pullback of this bundle along $f$.
Then, there is an induced $\ZN$-action on $Y$, and there is a canonical diffeomorphism of simplicial manifolds (i.e.~a morphism of simplicial manifolds which is a diffeomorphism in each degree)
\begin{equation}
	(f^*\RN) \dslash \ZN \longrightarrow \cC (f^*\RN)\,,
\end{equation}
where we have used the notation from Example~\ref{eg:M//G as spl Mfd}.
An element in $((f^*\RN) \dslash \ZN)_n = (f^*\RN) \times \ZN^n$ is a tuple $(x,r,k_1, \ldots, k_n)$, where $x \in M$, $r \in \RN$, and $k_1, \ldots, k_n \in \ZN$, and where $f(x) = \exp(2\pi\iu\, r)$.
The above map sends $(x,r,k_1, \ldots, k_n)$ to $((x,r), (x,r + k_1), \ldots, (x,r+k_n)) \in (f^*\RN)^{[n+1]}$.
We define a line bundle with connection $(p^*L)^{-\ZN}$ on $((f^*\RN) \dslash \ZN)_1 = (f^*\RN) \times \ZN$ by
\begin{equation}
	(p^*L)^{-\ZN}_{|(f^*\RN) \times \{k\}} =
	\begin{cases}
		(p^*L^\vee)^k\,, & k > 0\,,
		\\
		(p^*L)^{|k|}\,, & k \leq 0\,,
	\end{cases}
\end{equation}
where for any line bundle with connection $J$, we understand $J^0$ to be the trivial line bundle with trivial connection.
Further, we define an isomorphism of line bundles over $((f^*\RN) \dslash \ZN)_2$:
\begin{equation}
	\mu \colon d_0^*(p^*L)^{-\ZN} \otimes d_2^*(p^*L)^{-\ZN} \longrightarrow d_1^*(p^*L)^{-\ZN}\,,
\end{equation}
consisting of the canonical isomorphisms $L^{k_1}_{|x} \otimes L^{k_2}_{|x} \longrightarrow L^{k_1 + k_2}_{|x}$, for $x \in M$, $k_1, k_2 \in \ZN$.
Finally, we define $B \in \Omega^2(f^*\RN; \iu\RN)$ as $B = r \cdot p^*\curv(L)$ (recall that $((f^*\RN) \times \ZN)_0 = f^*\RN$).
Then, we have
\begin{equation}
	\curv \big( (p^*L)^{-\ZN} \big)_{(x,r,k_1)}
	= - k_1 \, p^*\curv(L)_{|(x,r)}
	= B_{|(x,r)} - B_{|(x,r+k_1)}
	= - (\delta B)_{|(x,r,k_1)}\,.
\end{equation}
Thus, $\CG = (p \colon f^*\RN \to M, (p^*L)^{-\ZN}, \mu, B)$ defines a bundle gerbe with connection on $M$, called the \emph{cup product gerbe} of $L$ and $f$.
Using the identification $\bbS^1 \cong \sfU(1)$, its curvature is
\begin{equation}
	\curv(\CG) = f^*\mu_{\sfU(1)} \wedge \curv(L)\,.
\end{equation}

\section{Holonomy, field theory, and strings}
\label{sec:PT and field theory}

Vector bundles with connection have a parallel transport along smooth paths, producing holonomies around smooth loops.
Locally, parallel transports are built directly from connection 1-forms.
Connections on gerbes consist of a 1-form and a 2-form (see Section~\ref{sec:Deligne coho and classification}); one could therefore expect connections on gerbes to have holonomies around one- and two-dimensional objects.
This is indeed the case; such holonomies and parallel transports have been investigated, for instance, in~\cite{Gawedzki:Topological_Actions, Gawedzki-Reis:WZW-branes_and_gerbes, CJM--Holonomy_on_D-branes, MP:Hol_and_PT_for_AbGrbs, BS:Higher_gauge_theory, SW:PT_and_functors, MP--2D_holonomy, Waldorf:PT_in_Pr2Buns, BMS:Smooth_2Grp_Ext_and_GrbSym}.
Here, we provide a modern, field-theoretic perspective on these holonomies on surfaces with and without boundaries, using the results of Sections~\ref{sec:Gerbes}.

\subsection{Surface Holonomy}
\label{sec:Surface Holonomy}

Let $M$ be a manifold carrying a gerbe $\CG$ with connection.
Let $\Sigma$ be a closed%
\footnote{A manifold $N$ is \emph{closed} if it is compact and satisfies $\partial N = \emptyset$.}
oriented surface, and let $\sigma \colon \Sigma \to M$ be a smooth map.
By Theorem~\ref{st:Grb(M) is 2Cat}, we can pull $\CG$ back along $\sigma$ to obtain a gerbe $\sigma^*\CG$ with connection on $\Sigma$.
As $\rmH^3(\Sigma;\ZN) = 0$, by Proposition~\ref{st:existence of trivs} there exists a trivialisation $\CT \colon \CI_\rho \to \sigma^*\CG$ (see also Definition~\ref{def:trivialisations of gerbes}).
Again by Proposition~\ref{st:existence of trivs}, for any other trivialisation $\CT' \colon \CI_{\rho'} \to \sigma^*\CG$, we have $\rho' - \rho \in \Omega^2_{\cl,\ZN}(\Sigma; \iu\RN)$.
Thus, the following is well-defined:

\begin{definition}
\label{def:surface hol}
Let $M$ be a manifold and let $\CG \in \Grb^\nabla(M)$.
Let $\Sigma$ be a closed, oriented surface, let $\sigma \colon \Sigma \to M$ be a smooth map, and let $\CT \colon \CI_\rho \to \sigma^*\CG$ be \emph{any} trivialisation of $\sigma^*\CG$.
The \emph{(surface) holonomy of $\CG$ around $(\Sigma,\sigma)$} is
\begin{equation}
\label{eq:surface hol}
	\hol(\CG;\sigma) = \exp \bigg( - \int_\Sigma \rho \bigg) \quad \in \sfU(1)\,.
\end{equation}
\end{definition}

This definition of surface holonomy goes back to~\cite{CJM--Holonomy_on_D-branes}, which succeeded earlier constructions~\cite{Witten:Current_algebra,Alvarez:Top_Quant_and_Coho,Gawedzki:Topological_Actions,Gawedzki-Reis:WZW-branes_and_gerbes}.
Definition~\ref{def:surface hol} makes full use of the 2-categorical theory of gerbes and does not rely on extensions of $\sigma$ to 3-manifolds $N$ with $\partial N = \Sigma$, or on combinatorial decompositions of $\Sigma$.
In physics, it describes the \emph{Wess-Zumino-Witten action}, which is part of the action of bosonic and fermionic strings.
One now readily proves:

\begin{proposition}
Let $M$ be a manifold and let $\CG \in \Grb^\nabla(M)$.
\begin{myenumerate}
\item If $N$ is a compact oriented 3-manifold with $\partial N = \Sigma$ and $f \colon N \to M$ is a smooth map with $f_{|\partial N} = \sigma$, then
\begin{equation}
	\hol(\CG;\sigma) = \exp \bigg( - \int_N \curv(\CG) \bigg)\,.
\end{equation}

\item $\CG$ has trivial holonomy around every closed surface $(\Sigma, \sigma)$ if and only if $\CG$ is flat, i.e.~$\curv(\CG) = 0$.
\end{myenumerate}
\end{proposition}

\begin{remark}
Finding a geometric origin for the Wess-Zumino-Witten action in conformal field theory was one of the driving forces behind the development of gerbes with connection, going back to~\cite{Witten:Current_algebra, Gawedzki:Topological_Actions}.
For further treatments of gerbes and D-branes from the perspective of conformal field theory, see, for instance,~\cite{Gawedzki-Reis:WZW-branes_and_gerbes, Waldorf--Thesis, FNSW:BGrbs_and_SurfHol} and references therein.
Finally, defects in conformal field theories and gauging of sigma-models are also described by gerbes~\cite{FSW:Bi-branes, GSW--Gauge_anomalies_in_2D-sigma-models, GSW--Gauging_sigma-models_with_defects}.
\qen
\end{remark}

\begin{remark}
There exists an extension of surface holonomy of gerbes to unoriented surfaces.
This requires a certain type of compatibility of the gerbe with the canonical involution on the orientation double cover of a surface; this compatibility is additional data on the gerbe called a \emph{Jandl structure}.
For references on Jandl gerbes, unoriented surface holonomy and applications to Wess-Zumino-Witten theory, we refer the reader to~\cite{SSW--Unoriented_WZW}.
\qen
\end{remark}

\subsection{Transgression of gerbes and holonomy on surfaces with boundary}
\label{sec:Transgression of Gerbes}

If $\Sigma$ is an oriented surface with $\partial \Sigma \neq \emptyset$, the surface holonomy of  $\CG \in \Grb^\nabla(M)$ around a smooth map $\sigma \colon \Sigma \to M$ is no longer well-defined as an element in $\sfU(1)$.
Instead, if $\CT \colon \CI_\rho \to \sigma^*\CG$ and $\CT' \colon \CI_{\rho'} \to \sigma^*\CG$ are two trivialisations, we have
\begin{equation}
	\hol(\CG;\sigma,\CT') = \hol(\CG;\sigma, \CT)\, \exp \bigg( - \int_\Sigma \rho' - \rho \bigg)\,.
\end{equation}
Here, $\hol(\CG;\sigma,\CT)$ is the surface holonomy of $\CG$ around $\sigma$ as in~\eqref{eq:surface hol}, computed with respect to $\CT$.
Observe that the error term on the right-hand side still vanishes if there is an isomorphism $\CT \cong \CT'$ of trivialisations (by Proposition~\ref{st:2isos as sections}).

The connected components of $\partial \Sigma$ are circles (up to orientation-preserving diffeomorphism).
To better understand what happens at the boundary, we thus consider smooth maps $\gamma \colon \bbS^1 \to M$.
By Proposition~\ref{st:existence of trivs} there exist trivialisations $\CT \colon \CI_0 \to \gamma^*\CG$.
(Here, we must have $\rho = 0$, since $\rho \in \Omega^2(\bbS^1; \iu\RN) = 0$.
It follows from Construction~\ref{cons:tensoring mps by VBuns}, Proposition~\ref{st:diffcoho hexagon}, and Proposition~\ref{st:existence of trivs} that any choice of trivialisation $\CT \colon \CI_0 \to \gamma^*\CG$ induces a canonical bijection
\begin{align}
	\pi_0 \big( \Triv(\gamma^*\CG) \big) \coloneqq \{\text{trivialisations of $\gamma^*\CG$}\} / \text{iso.}
	&\arisom \HLBdl^\nabla(\bbS^1) / \text{iso.} \eqqcolon \pi_0 \big( \HLBdl^\nabla(\bbS^1) \big)
	\cong \sfU(1)\,,
	\\*
\label{eq:ulHom in pi_0(Triv(G))}
	(\CT' \colon \CI_0 \to \gamma^*\CG) &\longmapsto \hol \big( \ul{\Hom}(\CT, \CT') \big)\,,
\end{align}
where the holonomy on the right-hand side is that of a line bundle with connection on $\bbS^1$.
At the same time, the tensor product from Construction~\ref{cons:tensoring mps by VBuns} induces an action of $\sfU(1) \cong \pi_0(\HLBdl^\nabla(\bbS^1))$ on $\pi_0(\Triv(\gamma^*\CG))$.
In fact, using Proposition~\ref{st:existence of trivs}, we derive

\begin{proposition}
\emph{\cite{Waldorf--Trangression_II}}
The set $\pi_0(\Triv(\gamma^*\CG))$ is a torsor over $\sfU(1) \cong \pi_0(\HLBdl^\nabla(\bbS^1))$.
\end{proposition}

To any $\sfU(1)$-torsor $P$, we can associate a complex line by viewing $P$ as a principal $\sfU(1)$-bundle over the point and performing the associated bundle construction.
In our situation, this yields the complex line
\begin{equation}
\label{eq:fib of CTCG at gamma}
	\CT\CG_{|\gamma} \coloneqq \pi_0 \big( \Triv(\gamma^*\CG) \big) \times_{\sfU(1)} \CN\,.
\end{equation}
An element in $\CT\CG_{|\gamma}$ is thus an equivalence class $[[\CT], z]$ of an isomorphism class $[\CT] \in \pi_0(\Triv(\gamma^*\CG))$ and a number $z \in \CN$.
Under the isomorphism $\pi_0(\HLBdl^\nabla(\bbS^1)) \cong \sfU(1)$, the equivalence relation reads as (compare Section~\ref{sec:spl view on LBuns})
\begin{equation}
	\big[ [\CT], z \big] = \big[ [\CT \otimes J],\, \hol(J)^{-1} z \big]\,,
\end{equation}
for $[J] \in \pi_0(\HLBdl^\nabla(\bbS^1))$.

Let $\Sigma$ be a compact oriented surface with a smooth map $\sigma \colon \Sigma \to M$.
Suppose the boundary of $\Sigma$ is partitioned into \emph{incoming} circles $\bbS^1_{in,a} \subset \partial\Sigma$, for $a = 1, \ldots, N_{in}$, and \emph{outgoing} circles $\bbS^1_{out,b} \subset \partial\Sigma$, for $b = 1, \ldots, N_{out}$.
We set
\begin{equation}
	\gamma_{in, a} \coloneqq \sigma_{|\bbS^1_{in,a}}
	\qandq
	\gamma_{out, b} \coloneqq \sigma_{|\bbS^1_{out, b}}\,.
\end{equation}
Suppose further that the incoming circles are endowed with the opposite orientation of that induced from $\Sigma$, and that the outgoing circles carry the orientation induced from $\Sigma$.
Consider a vector
\begin{equation}
	\xi = \bigotimes_{a = 1}^{N_{in}} \big[ [\CT_a],\, z_a \big]
	= z \cdot \bigotimes_{a = 1}^{N_{in}} \big[ [\CT_a],\, 1 \big]
	\ \in \bigotimes_{a = 1}^{N_{in}} \CT\CG_{|\gamma_{in, a}}\,,
\end{equation}
where $\CT_a \colon \CI_0 \to \gamma_{in, a}^*\CG$ is a trivialisation%
\footnote{Technically, in order to match the definition of $\CT\CG_{|\gamma}$, we have to choose diffeomorphisms $\bbS^1_{in,a} \cong \bbS^1$ and then check that the constructions are independent of that choice.
We will not go into these details here, but instead refer the reader to~\cite{BW:OCTFTs_and_Gerbes}.}.
Since $\sigma^*\CG$ is trivialisable over $\Sigma$, we may even choose a trivialisation $\CT \colon \CI_\rho \to \sigma^*\CG$ and assume that $\CT_a = \CT_{|\bbS^1_{in,a}}$, for $a = 1, \ldots, N_{in}$, where $\bbS^1_{in,a} \subset \Sigma$ is the $a$-th incoming boundary circle.
Define a new vector
\begin{equation}
\label{eq:Z[Sigma,sigma]}
	\scZ_\CG[\Sigma,\sigma;\CT](\xi)
	\coloneqq z \cdot \exp \bigg( - \int_\Sigma \rho \bigg) \cdot \bigotimes_{b = 1}^{N_{out}} \big[ [\CT_b],\, 1 \big]
	\ \in \bigotimes_{b = 1}^{N_{out}} \CT\CG_{|\gamma_{out, b}}\,,
\end{equation}
where we have set $\CT_b = \CT_{|\bbS^1_{out,b}}$.

\begin{proposition}
The vector $\scZ_\CG[\Sigma,\sigma;\CT](\xi)$ in~\eqref{eq:Z[Sigma,sigma]} is independent of the choice of trivialisation $\CT \colon \CI_\rho \to \sigma^*\CG$ over $\Sigma$.
\end{proposition}

\begin{proof}
Let $\CT' \colon \CI_{\rho'} \to \sigma^*\CG$ be another trivialisation.
By~\eqref{eq:ulHom in pi_0(Triv(G))} and~\eqref{eq:fib of CTCG at gamma}, we have
\begin{equation}
	\xi	= z \cdot \bigotimes_{a = 1}^{N_{in}} \big[ [\CT_a],\, 1 \big]
	= z \cdot \prod_{a=1}^{N_{in}} \hol \big( \ul{\Hom}(\CT_a,\CT'_a); \bbS^1_{in,a} \big)^{-1} \cdot \bigotimes_{a = 1}^{N_{in}} \big[ [\CT'_a],\, 1 \big]
	= z' \cdot \bigotimes_{a = 1}^{N_{in}} \big[ [\CT'_a],\, 1 \big]\,.
\end{equation}
Applying the construction~\eqref{eq:Z[Sigma,sigma]} to this, using the trivialisation $\CT'$ in place of $\CT$, we obtain
\begin{align}
	&\scZ[\Sigma,\sigma; \CT'](\xi)
	= z' \cdot \exp \bigg( - \int_\Sigma \rho' \bigg) \cdot \bigotimes_{b = 1}^{N_{out}} \big[ [\CT'_b],\, 1 \big]
	\\
	&= z \cdot \exp \bigg( - \int_\Sigma \rho' \bigg) \cdot \prod_{a=1}^{N_{in}} \hol \big( \ul{\Hom}(\CT_a,\CT'_a); \bbS^1_{in,a} \big)^{-1} \cdot \prod_{b=1}^{N_{out}} \hol \big( \ul{\Hom}(\CT'_b,\CT_b); \bbS^1_{out,b} \big)^{-1} \cdot \bigotimes_{b = 1}^{N_{out}} \big[ [\CT_b],\, 1 \big]
	\\
	&= z \cdot \exp \bigg( - \int_\Sigma \rho' \bigg) \cdot \prod_{a=1}^{N_{in}} \hol \big( \ul{\Hom}(\CT_a,\CT'_a); \bbS^1_{in,a} \big)^{-1} \cdot \prod_{b=1}^{N_{out}} \hol \big( \ul{\Hom}(\CT_b,\CT'_b); \bbS^1_{out,b} \big) \cdot \bigotimes_{b = 1}^{N_{out}} \big[ [\CT_b],\, 1 \big]
	\\
	&= z \cdot \exp \bigg( - \int_\Sigma \rho' \bigg) \cdot \hol \big( \ul{\Hom}(\CT,\CT'); \partial \Sigma \big) \cdot \bigotimes_{b = 1}^{N_{out}} \big[ [\CT_b],\, 1 \big]
	\\*
	&= z \cdot \exp \bigg( - \int_\Sigma \rho \bigg) \cdot \bigotimes_{b = 1}^{N_{out}} \big[ [\CT_b],\, 1 \big]\,.
\end{align}
Here we have used that $\ul{\Hom}(\CT,\CT')$ is the dual line bundle to $\ul{\Hom}(\CT',\CT)$ (which one can see from its construction, for instance---see Section~\ref{sec:Grbs and twVBuns}), Proposition~\ref{st:existence of trivs}(2), and the particular choices of orientations on the incoming and outgoing boundary circles.
\end{proof}

\begin{corollary}
\label{st:Z_G as linear map}
Let $\CG \in \Grb^\nabla(M)$.
For any compact, oriented surface $\Sigma$ whose boundary is partitioned into incoming and outgoing boundary components as above, and which is endowed with smooth map $\sigma \colon \Sigma \to M$, we obtain a linear map
\begin{equation}
	\scZ_\CG[\Sigma,\sigma] \colon \bigotimes_{a = 1}^{N_{in}} \CT\CG_{\gamma_{in, a}}
	\longrightarrow \bigotimes_{b = 1}^{N_{out}} \CT\CG_{|\gamma_{out, b}}\,.
\end{equation}
If $\partial \Sigma = \emptyset$, this reproduces the surface holonomy of $\CG$ from Section~\ref{sec:Surface Holonomy} as a linear map $\CN \to \CN$.
\end{corollary}

One can show that the linear maps $\scZ_\CG[\Sigma,\sigma]$ have various useful properties:
for instance, they depend only on the \emph{thin} homotopy class%
\footnote{Two smooth maps of manifolds $f,g \colon N \to M$ are \emph{thinly} homotopic if there exists a homotopy $h \colon [0,1] \times N \to M$ between them whose differential $h_*$ has rank at most $\dim(N)$ everywhere on $[0,1] \times N$.}
of $\sigma$, but at the same time depend on $\sigma$ smoothly in a precise sense.
Most importantly, they are compatible with gluing of surfaces along boundary components, and thus assemble into what is called a \emph{smooth functorial field theory on $M$} (in the sense of~\cite{ST:SuSy_FTs_and_generalised_coho}).
We refer the reader to~\cite{BW:Transgression_of_D-branes,BW:OCTFTs_and_Gerbes} for the full constructions and details;
here, we shall restrict ourselves to stating the main result:

\begin{theorem}
\label{st:FFTs from Grbs}
\emph{\cite{BW:OCTFTs_and_Gerbes}}
Any bundle gerbe $\CG \in \Grb^\nabla(M)$ with connection on $M$ gives rise to a two-dimensional, oriented, smooth functorial field theory on $M$.
This field theory depends functorially on $\CG$, and it admits an extension to an open-closed field theory in the presence of D-branes (see Section~\ref{sec:Rmks on pt and D-branes}).
\end{theorem}

This extends earlier results in this direction in~\cite{Gawedzki:Topological_Actions, BTW--Gerbes_and_HQFTs, Picken:TQFTs_and_Grbs}.
Let us illustrate some aspects of the results in Theorem~\ref{st:FFTs from Grbs}.
First, consider the complex line $\CT\CG_{|\gamma}$ from~\eqref{eq:fib of CTCG at gamma} we assigned to any smooth loop $\gamma$ in $M$.
Varying the curve $\gamma$, we have the following statement:
let $LM$ be the space of smooth maps $\gamma \colon \bbS^1 \to M$.
This is no longer a manifold, but one can describe it very conveniently as a \emph{diffeological spaces}~\cite{Iglesias-Zemmour:Diffeology,Baez-Hoffnung:Convenient_Spaces} (see also~\cite{Waldorf:Transgression_I, Bunk:Coh_Desc} for background on diffeological vector bundles).
The main idea behind diffeological spaces is to study spaces $N$ not by locally defined diffeomorphisms to euclidean space, but instead to use all maps from euclidean spaces to $N$ which satisfy a certain smoothness condition.
Many infinite-dimensional spaces which appear in geometry, such as mapping spaces and diffeomorphism groups of manifolds, can be naturally described as diffeological spaces.

\begin{theorem}
\emph{\cite{Waldorf--Trangression_II}}
The complex lines $\{\CT\CG_{|\gamma}\}_{\gamma \in LM}$ assemble into a diffeological hermitean line bundle $\CT\CG \to LM$ over the free loop space $LM$ of $M$.
\end{theorem}

\begin{definition}
The line bundle $\CT\CG$ on $LM$ is called the \emph{transgression line bundle} of $\CG \in \Grb^\nabla(M)$.
\end{definition}

The transgression line bundle $\CT\CG \to LM$ comes with a natural parallel transport, which can be seen as induced from the above field-theory construction:
a smooth path $\Gamma \colon [0,1] \to LM$ is equivalent to a smooth map $\Gamma^\dashv \colon [0,1] \times \bbS^1 \to M$.
By Corollary~\ref{st:Z_G as linear map} it gives rise to an isomorphism
\begin{equation}
	\PT^\CG_\Gamma \coloneqq \scZ_\CG \big[ [0,1] \times \bbS^1, \Gamma^\dashv \big]
	\colon \CT\CG_{|\Gamma(0)} \longrightarrow \CT\CG_{|\Gamma(1)}\,.
\end{equation}
One can show that this map is compatible with concatenation of paths%
\footnote{Technically, one needs to demand sitting instants normal to the boundary, or talk about cutting paths instead of concatenating.
Details for the first approach can be found in~\cite{BW:Transgression_of_D-branes,Waldorf--Trangression_II}.};
that is, for concatenatable paths $\Gamma_0, \Gamma_1 \colon [0,1] \to LM$, we have
\begin{equation}
	\PT^\CG_{\Gamma_1} \circ \PT^\CG_{\Gamma_0} = \PT^\CG_{\Gamma_1 * \Gamma_0}\,.
\end{equation}
The maps $\PT^\CG$ endow the line bundle $\CT\CG$ with a connection~\cite{Waldorf--Trangression_II}.
The parallel transport $\PT^\CG$ has several additional properties; it is compatible with \emph{fusion} of loops, depends only on the thin homotopy class of $\Gamma$, and for any two \emph{thin} paths $\Gamma_0, \Gamma_1 \colon [0,1] \to LM$ which agree at $0$ and $1$, we have $\PT^\CG_{\Gamma_0} = \PT^\CG_{\Gamma_1}$~\cite{Waldorf--Trangression_II, BW:Transgression_of_D-branes}.
Here, a path $\Gamma \colon [0,1] \to LM$ is thin if $\Gamma^\dashv_*$ has rank at most one everywhere on $[0,1] \times \bbS^1$.
Let $\HLBdl^\nabla_{fus}(LM)$ denote the symmetric monoidal groupoid of hermitean line bundles with connection on $LM$ which have the above additional properties.

\begin{theorem}
\emph{\cite{Waldorf--Trangression_II}}
There is an equivalence of symmetric monoidal groupoids
\begin{equation}
	\rmh_0 \big( \Grb^\nabla_{par\, iso}(M), \otimes \big)
	\simeq
	\big( \HLBdl^\nabla_{fus}(LM), \otimes \big)\,,
\end{equation}
where on the left-hand side $\Grb^\nabla_{par\, iso}(M)$ is the 2-groupoid of gerbes with connection on $M$ and only their parallel isomorphisms and 2-isomorphisms.
The $\rmh_0$ denotes the identification of 2-isomorphic isomorphisms, making the left-hand side into a groupoid.
\end{theorem}

\begin{remark}
It is possible to obtain from the transgression of gerbes a formalism of dimensional reduction, which turns a gerbe with connection over an oriented circle bundle $P \to M$ into a principal $\sfU(1)$-bundle $P' \to M$~\cite{Bunk--Thesis, BS:Fluxes_Gerbes_2HSpaces} (see~\cite{BW:OCTFTs_and_Gerbes} for the equivariance of $\CT\CG$ with respect to diffeomorphisms of $\bbS^1$).
We expect that this is part of an explicit construction of topological T-duality in the presence of generic $H$-flux (see e.g.~\cite{BEM:T-duality, NW:HGeo_for_NonGeo_T-duals, BN:T-Duality, Alfonsi--Global_DFT_and_KK, Alfonsi--Extended-higher_Review}).
\qen
\end{remark}

\subsection{Remarks on parallel transport and D-branes}
\label{sec:Rmks on pt and D-branes}

In this section we briefly survey some further aspects of the surface holonomy and transgression constructions for gerbes with connection.
The phrase \emph{surface holonomy} of a gerbe~\eqref{eq:surface hol} suggest that this is the \emph{holonomy} of a certain parallel transport.

In some sense, this is indeed the case:
if $\Sigma = \bbT^2 = (\bbS^1)^2$ is a 2-torus, then a smooth map $\sigma \colon \bbT^2 \to M$ is equivalent to a loop $\sigma^\vdash \colon \bbS^1 \to LM$, and we have
\begin{equation}
	\hol(\CG; \sigma) = \hol(\CT\CG; \sigma^\vdash)\,.
\end{equation}
Here, the holonomy on the left-hand side is the surface holonomy of the gerbe $\CG \in \Grb^\nabla(M)$, whereas on the right-hand side we have the (ordinary) holonomy of the transgression line bundle $\CT\CG \to LM$.

However, a full parallel transport on a gerbe $\CG \in \Grb^\nabla(M)$ should assign an isomorphism
\begin{equation}
	\PT^\CG_{1,\gamma} \colon \CG_{|\gamma(0)} \to \CG_{|\gamma(1)}
\end{equation}
to every smooth path $\gamma$ in $M$, such that $\PT^\CG_1$ depends smoothly on $\gamma$.
Furthermore, for any smooth homotopy $h \colon \gamma \Rightarrow \gamma'$ of paths, we should specify how the parallel transport changes; that is, any such homotopy should induce a 2-isomorphism
\begin{equation}
	\PT^\CG_{2,h} \colon \PT^\CG_{1,\gamma} \to \PT^\CG_{1,\gamma'}\,.
\end{equation}
These constituents of the parallel transport need to satisfy various further conditions regarding concatenations and thin homotopies.
This notion of parallel transport for gerbes with connection has been developed in~\cite{BMS:Smooth_2Grp_Ext_and_GrbSym}, extending the ideas in~\cite{BMS:NA_magnetic_translations}.

In particular, let $\gamma \colon \bbS^1 \to M$ be a smooth loop in $M$.
It induces an automorphism $\hol(\CG;\gamma) = \PT^\CG_{1,\gamma}$ of $\CG_{|\gamma(0)}$, which is a gerbe on the one-point manifold.
The category of automorphisms of $\CG_{|\gamma(0)}$ is canonically equivalent to the groupoid of one-dimensional hermitean vector spaces.
On can show that there is a canonical isomorphism~\cite[Prop.~4.19]{BMS:Smooth_2Grp_Ext_and_GrbSym}
\begin{equation}
	\hol(\CG;\gamma) \cong \CT\CG_{|\gamma}\,,
\end{equation}
of such vector spaces, which depends smoothly on $\gamma$ and functorially on $\CG$.
Thus, we can interpret the transgression line bundle of $\CG$ simply as the holonomy of the parallel transport $\PT^\CG$ of $\CG$.

Another extension of the transgression formalism and surface holonomy facilitates the inclusion of \emph{D-branes}.
The idea that D-branes and Chan-Paton bundles in string theory with non-trivial $B$-field are related to gerbes and their morphisms goes back to~\cite{Kapustin:D-branes_in_nontriv_B-fields}.
Further work on this geometric perspective on D-branes has been carried out in~\cite{CJM--Holonomy_on_D-branes, Gawedzki-Reis:WZW-branes_and_gerbes, Gawedzki:branes_in_WZW-models_and_gerbes, FSW:Bi-branes, FNSW:BGrbs_and_SurfHol}, for instance.

\begin{definition}
Let $\CG \in \Grb^\nabla(M)$.
A \emph{D-brane} for $\CG$ is a pair $(Q,\CE)$ of a submanifold $Q \subset M$ and a morphism $\CE \colon \CI_0 \to \CG_{|Q}$.
\end{definition}

Let $(Q_i,\CE_i)_{i \in \Lambda}$ be a collection of D-branes for $\CG$.
Using the 2-categorical theory of gerbes from Section~\ref{sec:Gerbes}, one can show that the transgression formalism extends from loops to open paths in the following way:
for each $i, j \in \Lambda$, let $P_{ij}M$ be diffeological space $P_{ij}M$ of smooth paths in $M$ with sitting instants which start on $Q_i$ and end on $Q_j$.
From $\CG$ and $(Q,\CE)$ one can construct a hermitean vector bundle with connection $\scR_{ij} \to P_{ij}M$~\cite{BW:Transgression_of_D-branes}.
These vector bundles come with various structure morphisms related to operations on the level of paths:
for instance, there is a canonical linear map
\begin{equation}
	\scR_{jk|\gamma_1} \otimes \scR_{ij|\gamma_0} \longrightarrow \scR_{ik|\gamma_1 * \gamma_0}
\end{equation}
for each $\gamma_0 \in P_{ij}M$ and $\gamma_1 \in P_{jk}M$ such that their concatenation $\gamma_1 * \gamma_0$ is defined.
These linear maps assemble into a smooth morphism of diffeological vector bundles.

In fact, for any fixed collection of submanifolds $\{Q_i \subset M\}_{i \in \Lambda}$, the vector bundles $\scR_{ij}$ and their structure morphisms depend functorially on the D-branes $\CE_i$ supported on $Q_i$, and one can even reconstruct the gerbe $\CG$ and the D-branes $\CE_i$---up to canonical isomorphism---from just knowing the transgression line bundle $\CT\CG \to LM$, the bundles $\scR_{ij} \to P_{ij}M$, and their structure morphisms.
We refer the reader to~\cite{BW:Transgression_of_D-branes} for the full statement and proof.
From a physical perspective, this makes precise how closed strings in $M$ and open strings stretched between D-branes in $M$ can detect the $B$-field and the twisted Chan-Paton bundles on the D-brane world volumes.

\section{Higher geometric prequantisation}
\label{sec:Higher Geo Quan}

Let $(M,\omega)$ be a symplectic manifold.
Geometric quantisation of $(M,\omega)$ relies, first of all, on a realisation of $\iu\omega$ as the curvature of a connection on a hermitean line bundle $L$ on $M$.
The (compactly supported) square-integrable sections of $L$ then form the prequantum Hilbert space of the system.
Kostant-Souriau prequantisation sends functions $f \in C^\infty(M)$ to operators $O_f$ on this Hilbert space, acting as $O_f = \iu \hbar \nabla^L_{X_f} (-) + f \cdot (-)$.
Here, $X_f$ is the Hamiltonian vector field of $f$.
This, however, does not represent the commutative algebra $C^\infty(M)$ on the prequantum Hilbert space, but rather $C^\infty(M)$ endowed with the (rescaled) Poisson bracket $\iu \hbar \{-,-\}$ induced by $\omega$.

In many geometric situations, symplectic forms are absent (for instance, on any odd-dimensional manifold).
However, related features might still be present.
For example, while the 2-sphere is symplectic, the 3-sphere is not, but instead it carries a closed 3-form which is non-degenerate in a certain sense (see below).
The curvature of $\CG \in \Grb^\nabla(M)$ is a closed 3-form $\curv(\CG)$ on $M$ with integer cycles, and an analogue of the first step in geometric prequantisation of 3-forms should be to realise a 3-form as the curvature of a gerbe with connection, instead of a line bundle.
However, there are two different concepts of what a \emph{higher symplectic form} should be; we will survey both of these and show how gerbes fit into both frameworks.

\subsection{Geometric prequantisation of 3-plectic forms}
\label{sec:3-plectic prequan}

In this section, we recall parts of the theory of \emph{$n$-plectic forms}, focussing on the case of $n = 2$.
For background and details we refer the reader to~\cite{CIL:Geo_of_multisymplectic_mfds, Rogers:Thesis, Rogers:2-plectic_geo_Corant_prequan, FRS:L_infty-algs_of_local_obs, FRS:Higher_U(1)-gerbe_connections}.

\begin{definition}
An \emph{$n$-plectic form} on a manifold $M$ is a closed $(n{+}1)$-form $\omega$ which is non-degenerate, meaning that the map $\iota_{(-)} \omega \colon TM \to \Lambda^nT^*M$, $X \mapsto \omega(X,-,\cdots,-)$ is injective.
\end{definition}

\begin{definition}
Let $(M,\omega)$ be a $2$-plectic manifold.
A \emph{prequantum bundle gerbe} for $(M,\omega)$ is a gerbe $\CG \in \Grb^\nabla(M)$ with connection on $M$ such that $\curv(\CG) = \iu \omega$.
We call the choice of such a gerbe with connection a \emph{prequantisation} of $(M,\omega)$.
\end{definition}

\begin{example}
\label{eg:G as 2-pl mfd}
Let $G$ be a compact, simple, simply connected Lie group with Lie algebra $\frg$.
The Killing form $\<-,-\>$ and commutator on $\frg$ induce a closed 3-form $\omega_3 = \frac{1}{6} \<-, [-,-]\>$ on $G$.
This form is 2-plectic and admits a prequantisation, given by the so-called \emph{basic} gerbe~\cite{Meinrenken:Basic_Gerbe}, or the \emph{tautological gerbe}~\cite{Murray:Bundle_gerbs}.
\qen
\end{example}

It follows from Theorem~\ref{st:classification of Grbs} and Proposition~\ref{st:diffcoho hexagon} that a 2-plectic manifold $(M,\omega)$ admits a prequantisation if and only if $\iu\, \omega \in \Omega^3_{cl,\ZN}(M; \iu\RN)$.
The construction of the Poisson algebra of functions in the symplectic case relies on the notion of Hamiltonian vector fields:
given $f \in C^\infty(M)$, a Hamiltonian vector field for $f$ is a vector field $X_f \in \Gamma(M; TM)$ such that $\iota_{X_f} \omega = \dd f$.
In the $n$-plectic case, Hamiltonian vector fields cannot be associated to functions, but to (certain) $(n{-}1)$-forms:

\begin{definition}
Let $(M,\omega)$ be an $n$-plectic manifold.
A \emph{Hamiltonian $n$-form} is an $(n{-}1)$-form $\eta \in \Omega^n(M)$ such that there exists a vector field $X_\eta \in \Gamma(M; TM)$ with $\iota_{X_\eta} \omega = \dd \eta$.
We denote the vector space of Hamiltonian $(n{-}1)$-forms on $M$ by $\Omega^{n-1}_\Ham(M)$.
\end{definition}

Note that for $n>2$ the map $\iota_{(-)}\omega \colon TM \to \Lambda^n T^*M$ is generally not surjective, so that $\Omega^{n-1}_\Ham(M)$ is, in general, a proper subspace of $\Omega^{n-1}(M)$.
However, since the map is injective, it follows that, for each $\eta \in \Omega^{n-1}_\Ham(M)$, the vector field $X_\eta$ with $\iota_{X_\eta} \omega = \dd \eta$ is unique.
We hence call $X_\eta$ the \emph{Hamiltonian vector field of $\eta$}.
The $n$-plectic version of the Poisson algebra of smooth functions on $M$ is given as follows:
we first recall the definition of $L_\infty$-algebras (also called strongly homotopy Lie algebras)~\cite{LS:Intro_to_SH_LieAlgs, LM:Strongly_Ho_LieAlgs}, see also~\cite[Def.~3.7]{Rogers:Thesis}.

\begin{definition}
An $L_\infty$-algebra is a $\ZN$-graded vector space $L$ with a collection $\{ l_k \colon L^{\otimes k} \to L\, | \, k \in \NN \}$ of skew-symmetric linear maps of degree $|l_k| = k-2$, satisfying the identity
\begin{equation}
\label{eq:L_oo identities}
	\sum_{i + j = m+1} \sum_{\sigma \in \UnSh(i, m-i)} (-1)^\sigma \epsilon(\sigma)\, (-1)^{i(j-1)}
	l_j \big( l_i(v_{\sigma(1)}, \ldots, v_{\sigma(i)}), v_{\sigma(i+1)}, \ldots, v_{\sigma(m)} \big)
	= 0
\end{equation}
for every $m \in \NN$.
Here, $\UnSh(i,j)$ is the set of $(i,j)$-unshuffles, where $\epsilon(\sigma)$ is the Koszul sign arising from applying the permutation $\sigma \in \UnSh(i,j)$ to the vectors $v_1, \ldots, v_{i+j}$, and where $(-1)^\sigma$ is the degree of the permutation $\sigma$.
A \emph{Lie $n$-algebra} is an $L_\infty$-algebra whose underlying graded vector space is concentrated in degrees $0, \ldots, n-1$ (in that case we obtain $l_k = 0$ for all $k > n+1$).
\end{definition}

One can check that $l_1 \eqqcolon \dd$ is a differential, turning $L$ into a chain complex, and that it is a graded derivation with respect to the bracket $l_2 \eqqcolon [-,-]$.
This bracket, however, does not satisfy the Jacobi identity; instead the Jacobi identity is violated up to a coherent set of homotopies.
Morphisms of $L_\infty$-algebras are more intricate to define; instead of doing this directly on the level of $L_\infty$-algebras, one usually passes to the coalgebra description of $L_\infty$-algebras:
if $L$ is a graded vector space, an $L_\infty$-algebra structure is equivalent to a choice of a codifferential on the (non-unital) coalgebra $\bigvee^\bullet L[1]$ of symmetric tensor powers, and morphisms of $L_\infty$-algebras are most elegantly described as morphisms of the associated codifferential coalgebras; see~\cite[Appendix~A]{JRSW:L_oo_and_BV} for a fully detailed account.

\begin{example}
\label{eg:Lie-2 algebra}
A Lie 2-algebra consists of a 2-term chain complex $L_1 \to L_0$, a skew-symmetric bracket $l_2 = [-,-] \colon L \otimes L \to L$ and the skew-symmetric Jacobiator $l_3 = J(-,-,-) \colon L^{\otimes 3} \to L$, which is precisely a chain homotopy of maps $L^{\otimes 3} \to L$ from $x \otimes y \otimes z \mapsto [x, [y,z]]$ to the map $x \otimes y \otimes z \mapsto [[x,z], z] + [y, [x,z]]$, for $x,y,z \in L_0$.
Finally, there is a compatibility relation between $[-,-]$ and $J$ (c.f~\cite[Eq.~3.10]{Rogers:Thesis}).

A morphism of Lie 2-algebras $L \to L'$ is a morphism of complexes $\phi \colon L \to L'$ and a chain homotopy $\Phi$ of maps $L \otimes L \to L'$ from $x \otimes y \mapsto \phi([x,y])$ to $x \otimes y \mapsto [\phi(x), \phi(y)]$, satisfying a compatibility relation (see~\cite[Eq.~3.11]{Rogers:Thesis}).
Such a morphism is called a \emph{quasi-isomorphism} if $\phi$ induces isomorphisms between the homology groups of the complexes $L$ and $L'$.
For full details, see~\cite{BC:Lie_2-algs} or~\cite[Sec.~3.2.1]{Rogers:Thesis}, for instance.
\qen
\end{example}

\begin{theorem}
\emph{\cite[Thm.~3.14]{Rogers:Thesis}}
Let $(M,\omega)$ be an $n$-plectic manifold.
There exists a Lie $n$-algebra $L_\infty(M,\omega)$ with
\begin{myitemize}
\item underlying graded vector space given by $L_0 = \Omega^{n-1}_\Ham(M)$, $L_i = \Omega^{n-1-i}(M)$ for $i = 1, \ldots n-1$, and $L_i = \{0\}$ otherwise,

\item differential $l_1 = \dd$ given by the de Rham differential on $L_i$ with $i >0$, and

\item higher brackets given by
\begin{equation}
	l_k(\alpha_1, \ldots, \alpha_k) =
	\begin{cases}
		0\,, & |\alpha_1 \otimes \cdots \otimes \alpha_k| > 0\,,
		\\
		(-1)^{\frac{k}{2}+1} \iota_{X_{\alpha_1} \wedge \cdots \wedge X_{\alpha_k}} \omega\,, & |\alpha_1 \otimes \cdots \otimes \alpha_k| = 0,\, k \text{ even}\,,
		\\
		(-1)^{\frac{k-1}{2}} \iota_{X_{\alpha_1} \wedge \cdots \wedge X_{\alpha_k}} \omega\,, & |\alpha_1 \otimes \cdots \otimes \alpha_k| = 0,\, k \text{ odd}\,.
	\end{cases}
\end{equation}
\end{myitemize}
\end{theorem}

\begin{definition}
For an $n$-plectic manifold $(M,\omega)$, we call $L_\infty(M,\omega)$ the \emph{Poisson Lie $n$-algebra} assocaited to $(M,\omega)$.
\end{definition}

In the case where $(M,\omega)$ is a $2$-plectic manifold which admits a prequantisation $\CG \in \Grb^\nabla(M)$, one can associate to it the Lie 2-algebra of infinitesimal symmetries of the gerbe with connection $\CG$.
It has recently been proven by Krepski and Vaughan~\cite{KV:Multiplicative_VFs} that this Lie 2-algebra is equivalent to the Poisson Lie 2-algebra of $(M,\omega)$, thus giving $L_\infty(M,\omega)$ a geometric description in terms of vector fields on the prequantum gerbe.
Similar results, though in a less geometric and more homotopy-theoretic flavour, have been obtained in~\cite{FRS:L_infty-algs_of_local_obs, FRS:Higher_U(1)-gerbe_connections}.

The explicit description of infinitesimal symmetries of gerbes in terms of local data goes back to~\cite{Collier:Sym_of_gerbes, FRS:L_infty-algs_of_local_obs, FRS:Higher_U(1)-gerbe_connections} and has recently been recast in global terms and the language of bundle gerbes in~\cite{KV:Multiplicative_VFs}.
This uses the theory of multiplicative vector fields on Lie groupoids, introduced in~\cite{MX:Lifting_and_multi_VFs}.
In particular, given a bundle gerbe $\CG = (\pi \colon Y \to M, L, \mu)$, we will from now on trade the line bundle $L$ for its underlying $\sfU(1)$-bundle, which we denote by $P$ (note that this neither loses nor adds information).
The structure of the bundle gerbe gives rise to smooth maps $s, t \colon P \to Y$ and $s_0 \colon Y \to P$, and together with $\mu \colon d_0^*P \otimes d_2^*P \to d_1^*P$, we obtain a Lie groupoid $(P \rightrightarrows Y)$ from these data.
We shall not describe multiplicative vector fields on Lie groupoids in full generality here, but restrict ourselves to the specific case of multiplicative vector fields on gerbes.

\begin{definition}
\cite[Prop.~3.9, Cor.~3.18]{KV:Multiplicative_VFs}
Let $\CG = (\pi \colon Y \to M, P, \mu)$ be a gerbe on $M$.
A \emph{multiplicative vector field on $\CG$} is a pair $\xi = (\xi_0, \xi_1)$, where $\xi_0 \in \Gamma(Y; TY)$ and $\xi_1 \in \Gamma(P; TP)$ satisfying that $\xi_1$ is $\sfU(1)$-invariant, that $d_{i*}\xi_1 = \xi_0$ for $i = 0,1$, and that
\begin{equation}
	\mu_{*|(y_0,y_1,y_2)}(\xi_{1|(y_1,y_2)} \otimes \xi_{1|(y_0,y_1)}) = \xi_{1|(y_0,y_2)}
\end{equation}
for all $(y_0,y_1,y_2) \in Y^{[3]} = \cC Y_2$.
\end{definition}

Note that $\xi = (\xi_0,\xi_1)$ is denoted $(\widetilde{\xi}, \check{\xi})$ in~\cite{KV:Multiplicative_VFs}.
We now consider connection-preserving multiplicative vector fields on $\CG \in \Grb^\nabla(M)$.
The connection on $\CG$ consists of a connection 1-form $A \in \Omega^1(P; \iu\RN)$ and a curving $B \in \Omega^2(Y; \iu\RN)$.
A multiplicative vector field $\xi$ on $\CG$ \emph{is connection preserving} if there exists $\alpha \in \Omega^1(Y; \iu\RN)$ such that
\begin{equation}
	(\pounds_{\xi_1}A, \pounds_{\xi_0}B) = (\dd \alpha, p^*\delta \alpha) \eqqcolon \rmD \alpha\,,
\end{equation}
where $\pounds$ denotes the Lie derivative.
This can be seen as the requirement that $(\pounds_{\xi_1}A, \pounds_{\xi_0}B)$ be exact in the complex obtained by applying Construction~\ref{cons:totalisation} to the two-term complex of sheaves $\Omega^1 \to \Omega^2$ and the simplicial manifold $\cC Y$.
Krepski and Vaughan then define an appropriate notion of morphisms between such connection-preserving vector fields, following~\cite{BEL:Lie_2-algs_of_VFs}; these are obtained as certain sections of the Lie algebroid associated to the Lie groupoid $(P \rightrightarrows Y)$; for details we refer to~\cite{KV:Multiplicative_VFs}.
There also exists a nice treatment of multiplicative vector fields and their relation to 2-plectic quantisation in~\cite{SW--PreQuan_for_2plectic_Mfds}, where the authors work specifically with lifting gerbes for projective unitary bundles on $M$.

\begin{proposition}
\emph{\cite[Cor.~3.18, Prop.~4.8]{KV:Multiplicative_VFs}}
Let $\CG \in \Grb^\nabla(M)$.
There exists a Lie 2-algebra $\bbX^\nabla(\CG)$ whose level-zero part is the vector space of connection-preserving vector fields on $\CG$, with bracket
\begin{equation}
	\big[ (\xi_0, \xi_1, \alpha),\, (\zeta_0, \zeta_1, \beta) \big]
	= \big( [\xi_0, \zeta_0],\, [\xi_1, \zeta_1],\, \pounds_{\xi_0} \beta - \pounds_{\zeta_0} \alpha \big)\,.
\end{equation}
\end{proposition}

The geometric interpretation of the Poisson Lie 2-algebra associated to the 2-plectic manifold $(M,\omega)$ is the following result.
It uses the notion of \emph{butterflies} between Lie 2-algebras.
These provide a weak notion of morphisms of Lie 2-algebras which describes the localisation of the 2-category of Lie 2-algebras at the quasi-isomorphisms; essentially, the existence of an invertible butterfly $L \to L'$ means that there is a finite chain $L \leftarrow J_0 \rightarrow J_1 \leftarrow J_2 \rightarrow \cdots \leftarrow J_n \rightarrow L'$ of quasi-isomorphisms of Lie 2-algebras.
For details, see~\cite{Noohi:Integrating_mps_of_Lie-2-algs}, where this theory was developed.

\begin{theorem}
\emph{\cite[Thm.~5.1]{KV:Multiplicative_VFs}}
Let $(M,\omega)$ be a 2-plectic manifold with prequantisation $\CG \in \Grb^\nabla(M)$.
Then, there is an invertible butterfly of Lie 2-algebras $L_\infty(M,\omega) \to \bbX^\nabla(\CG)$.
\end{theorem}

\begin{remark}
There is also a relation between certain infinitesimal symmetries of gerbes with connection and the Lie 2-algebra of sections of the Courant algebroid associated to $(M,\omega)$.
This was first described in works of Rogers~\cite{Rogers:Thesis} and worked out later in more detail in~\cite{Collier:Sym_of_gerbes, FRS:L_infty-algs_of_local_obs, FRS:Higher_U(1)-gerbe_connections}.
Such a relation is expected from the link between gerbes and Courant algebroids in generalised geometry observed already in~\cite{Hitchin:Special_Lagr_Lectures, Hitchin:Generalised_CY, Gualtieri:Thesis}, for instance.
\qen
\end{remark}

\begin{remark}
In this section, we have focussed purely on the \emph{observables} in geometric prequantisation of 2-plectic manifolds.
It is a different---but related---problem to describe the \emph{states} in this theory.
Since the line bundle of geometric prequantisation is replaced by a gerbe with connection $\CG$ on $M$, one should expect the states to consist of sections of $\CG$.
This idea was first investigated in~\cite{Rogers:Thesis}, and a categorified Hilbert space of sections was constructed in~\cite{BSS--HGeoQuan, Bunk--Thesis}.
However, these sections---and even more so how the observables act on them---are still not understood well enough, and there is broad scope for further research.
Let us point out the recent paper~\cite{Safronov:Shifted_GeoQuan}, which also makes progress in this direction.
\qen
\end{remark}

\subsection{Shifted symplectic forms}
\label{sec:shifted symplectic}

Finally, we illustrate another approach to higher-degree generalisations of symplectic manifolds, going by the name of \emph{shifted symplectic structures}.
Their introduction in~\cite{PTVV:Shifted_symplectic_structures} has lead to significant advances in (derived) algebraic geometry.
In differential geometry, (1-)shifted symplectic forms have so far mostly appeared in the study of quasi-symplectic groupoids~\cite{BCWZ--Integration_of_twisted_Dirac, Xu:Momentum_and_Morita, LGX--Pre-quasi-sym_quant_via_Grbs}, but see~\cite{Pridham:ShPois_in_DerDiffGeo} for a perspective from derived differential geometry.

Here, we consider shifted symplectic forms on simplicial manifolds~\cite{Getzler:Slides_on_Stacks}.
Let $X = \{X_k, d_i, s_i\}$ be a simplicial manifold (cf.~Section~\ref{sec:spl view on LBuns}).
If a 2-form $\omega_2$ on $X_k$ is not closed, its failure to be so could be an exact term with respect to the \v{C}ech differential, i.e.~there could be a 3-form $\omega_3$ on $X_{k-1}$ such that $\dd \omega_2 = \delta \omega_3$.
The 3-form $\omega_3$ could now again fail to be closed up to a \v{C}ech-exact term, and so on.
The central idea for shifted symplectic forms is to replace closed 2-forms by 2-forms closed up to a coherent chain of such higher-degree forms.
Making this rigorous relies simply on Construction~\ref{cons:totalisation}.
In the presentation of this material, we heavily draw from~\cite{Getzler:Slides_on_Stacks}.

For $k \in \NN_0$, consider the cochain complex of sheaves of abelian groups
\begin{equation}
	\tau_{\geq k} \Omega^\bullet[k] = \big(
	\begin{tikzcd}
		0 \ar[r] & \Omega^k \ar[r, "\dd"] & \Omega^{k+1} \ar[r, "\dd"] & \Omega^{k+2} \ar[r, "\dd"] & \cdots
	\end{tikzcd}
	\big)\,.
\end{equation}
Note that $\Omega^k$ sits in degree zero in this complex.
We remark that $\tau_{\geq k} \Omega^\bullet[k]$ is an injective resolution of the sheaf $\Omega^k_\cl$ of \emph{closed} $k$-forms.

\begin{definition}
\cite{PTVV:Shifted_symplectic_structures, Getzler:Slides_on_Stacks}
The \emph{complex of closed $k$-forms} on a simplicial manifold $X$ is
\begin{equation}
	\CA^k_\cl(X) = \Tot \big( \tau_{\geq k} \Omega^\bullet[k](X) \big)\,.
\end{equation}
A \emph{closed $k$-form of degree $p$ on $X$} is a degree-$p$ cocycle $\omega \in Z^p(\CA^k_\cl(X))$.
\end{definition}

Explicitly, we obtain from Construction~\ref{cons:totalisation} that a closed $k$-form of degree $p$ on $X$ is a $p$-tuple $\omega = (\omega_{p+k}, \omega_{p+k-1}, \ldots, \omega_{k+1}, \omega_k)$ with $\omega_{k+i} \in \Omega^{k+i}(X_{p-i})$, satisfying $\rmD \omega = 0$, i.e.
\begin{align}
	\dd \omega_{p+k} &= 0\,,
	\\
	\dd \omega_{k+i+1} + (-1)^i \delta \omega_{k+i} &= 0\,, \qquad \text{for } i = 0, \ldots, p-1\,,
	\\
	\delta \omega_k &= 0\,.
\end{align}

\begin{example}
\label{eg:closed (2,2)-form on BG}
Let $G$ be a compact, simple, simply connected Lie group.
Recall the simplicial manifold $\rmB G$ from Example~\ref{eg:BG as spl Mfd}.
Further, recall the closed 3-form $\omega_3 \in \Omega^3(G)$ from Example~\ref{eg:G as 2-pl mfd}.
Let $\mu_G$ be the (left-invariant) Maurer-Cartan form on $G$, let $\overline{\mu}_G$ be the right-invariant Maurer-Cartan form on $G$, and define the 2-form
\begin{equation}
\label{eq:omega_2 on BG}
	\omega_2 = \frac{1}{2} \<d_2^*\mu_G, d_0^*\overline{\mu}_G\>
\end{equation}
on $G^2 = \rmB G_2$.
Here, $d_2$ and $d_0$ are the face maps of the simplicial manifold $\rmB G$.
Then,
\begin{equation}
	\omega = (0, \omega_3, \omega_2) \in Z^2 \big( \CA^2_\cl(\rmB G) \big)
\end{equation}
is a closed 2-form of degree two on $\rmB G$ (see, for instance,~\cite{Waldorf:Multiplicative_Gerbes}).
\qen
\end{example}

For a simplicial manifold $X$, we can further define a tangent bundle (or tangent complex) in the following sense:
the tangent bundles $\{TX_k \to X_k\}_{k \in \NN_0}$  each pull back to $X_0$ along the compositions $s_0 \circ \cdots \circ s_0 \colon X_0 \to X_k$.
As a consequence of the simplicial identities~\eqref{eq:spl identities}, the differentials of the face and degeneracy maps of $X$ induce on the collection of these pullbacks the structure of a simplicial vector bundle $\bbT^sX$ on $X_0$; that is, $\bbT^sX$ is a collection $\{\bbT^s_k X \to X_0\}_{k \in \NN_0}$ of vector bundles on $X_0$, endowed with morphisms of vector bundles $\partial_i \colon \bbT^s_k X \to \bbT^s_{k-1}X$ and $\sigma_i \colon \bbT^s_k X \to \bbT^s_{k+1} X$ which satisfy the simplicial identities~\eqref{eq:spl identities} (i.e.~$\bbT^sX$ is a simplicial object in the category of vector bundles on $X_0$).
We can now apply a dual version of Construction~\ref{cons:AltFace coch complex for csp VSp} to $\bbT^s X$ (dual in the sense that cosimplicial objects are replaced by simplicial ones, and cochain complexes by chain complexes) to obtain a chain complex
\begin{equation}
	\bbT X = \big(
	\begin{tikzcd}
		0 & \bbT^s_0 X \ar[l]
		& \bbT^s_1 X \ar[l, "\Delta"']
		& \bbT^s_2 X \ar[l, "\Delta"']
		& \cdots \ar[l, "\Delta"']
	\end{tikzcd}
	\big)\,,
\end{equation}
where $\Delta \colon (\bbT X)_k \to (\bbT X)_{k-1}$ is given by $\Delta = \sum_{i = 0}^k (-1)^i \partial_i$.
We remark that the construction of $\bbT X$ given in~\cite{Getzler:Slides_on_Stacks} is not isomorphic to our construction here, but it is canonically \emph{quasi-}isomorphic to our definition (this is due to the quasi-isomorphism between the normalised chain complex and the Moore complex associated to a simplicial abelian group~\cite[Thm.~III 2.4]{GJ:Spl_HoThy}).

\begin{definition}
\label{def:tangent complex}
Let $X$ be a simplicial manifold.
The chain complex $(\bbT X, \Delta)$ of vector bundles on $X_0$ is called the \emph{tangent complex of $X$}.
\end{definition}

Let $\omega = (\omega_{2+p}, \ldots, \omega_2)$ be a closed 2-form of degree $p$ on a simplicial manifold $X$.
Consider two elements $\xi, \xi'$ in the (fibre of the) tangent complex $\bbT_{|x} X$ of $X$ at $x \in X_0$ whose degrees $|\xi| \eqqcolon a$ and $|\xi| \eqqcolon b$ satisfy $a + b = p$.
We define the pairing
\begin{equation}
\label{eq:pairing from derived 2-form}
	\omega(\xi, \xi') \coloneqq \sum_{\varrho \in \Sh(a,b)} (-1)^\varrho\, \omega_2 \big( (\sigma_{\varrho(n-1)} \circ \cdots \circ \sigma_{\varrho(a)})_* \xi,\, (\sigma_{\varrho(a-1)} \circ \cdots \circ \sigma_{\varrho(0)})_* \xi' \big)\,,
\end{equation}
where $\Sh(a,b)$ is the set of $(a,b)$-shuffles.
One can now show that this pairing is (graded) antisymmetric, and that $\Delta$ is (graded) self-adjoint with respect to $\omega$.
In particular, the pairing~\eqref{eq:pairing from derived 2-form} induces a pairing of degree two on the homology $\rmH_\bullet(\bbT X_{|x}, \Delta)$ for each $x \in X_0$~\cite{Getzler:Slides_on_Stacks}.

\begin{remark}
The explicit form for the pairing~\eqref{eq:pairing from derived 2-form} and its properties arise form the Eilenberg-Zilber map, which induces a quasi-isomorphism $\bbT X \otimes \bbT X \to \bbT X \tilde{\otimes} \bbT X$ between the usual tensor product $\bbT X \otimes \bbT X$ of chain complexes and the level-wise tensor product, whose level-$k$ vector space is $(\bbT X \tilde{\otimes} \bbT X)_k = \bbT_k X \otimes \bbT_k X$; for background on the Eilenberg-Zilber map and its properties, see, for instance,~\cite[Sec.~29]{May:Simplicial_Objects}.
\qen
\end{remark}

\begin{definition}
Let $X$ be a symplectic manifold.
A \emph{$p$-shifted symplectic form on $X$} is a closed 2-form $\omega$ of degree $p$ on $X$ for which the pairing~\eqref{eq:pairing from derived 2-form} induces a non-degenerate pairing on the homology of $\bbT X$ (at every point $x \in X_0$).
A \emph{$p$-shifted symplectic simplicial manifold} is a pair $(X,\omega)$ of a simplicial manifold $X$ and a $p$-shifted symplectic form $\omega$ on $X$.
If $Y \to M$ is a surjective submersion, we say that a $p$-shifted symplectic form $\omega$ on the \v{C}ech nerve $\cC Y$ is a \emph{$p$-shifted symplectic form on $M$}.
\end{definition}

\begin{example}
\label{eg:BG as 2-shifted symplectic}
Consider the simplicial manifold $\rmB G$ and its closed 2-form $\omega = (0,\omega_3, \omega_2)$ of degree two from Example~\ref{eg:closed (2,2)-form on BG}.
The manifold $(\rmB G)_0 = *$ consists of a single point.
Thus, the tangent complex $\bbT \rmB G$ is a chain complex of vector bundles on the point; that is, it is simply a chain complex of real vector spaces.
We find that $(\bbT \rmB G)_k = \frg^{k-1}$, where $\frg$ is the Lie algebra of $G$.
It remains to understand the differential on $\bbT \rmB G$.
Since the differential of the multiplication $G^2 \to G$ at the neutral element $e \in G$ is simple the addition in $\frg$, one obtains the following explicit expressions:
\begin{alignat}{3}
	&\Delta \colon \frg^2 \to \frg\,,
	&& \quad (\xi_1, \xi_2) &&\longmapsto \xi_1 - (\xi_1 + \xi_2) + \xi_2 = 0\,,
	\\
	&\Delta \colon \frg^3 \to \frg^2\,,
	&& \quad (\xi_1, \xi_2, \xi_3) &&\longmapsto (\xi_2,\xi_3) - (\xi_1 + \xi_2, \xi_3) + (\xi_1, \xi_2 + \xi_3) - (\xi_1, \xi_2)
	= (-\xi_1, \xi_3)\,,
	\\
	&\Delta \colon \frg^4 \to \frg^3\,,
	&& \quad (\xi_1, \xi_2, \xi_3, \xi_4) &&\longmapsto
	(0, \xi_2 + \xi_3, 0)
\end{alignat}
and so on.
Let $\frg[1]$ denote the chain complex with $\frg$ in degree one and all other degrees trivial.
The morphism $\frg[1] \to \bbT \rmB G$, $\xi \mapsto \xi$ is a quasi-isomorphism, inducing
\begin{equation}
	\rmH_\bullet(\frg[1], 0) \arisom \rmH_\bullet(\bbT \rmB G, \Delta)\,.
\end{equation}
Finally, consider the pairing induced by $\omega$.
Since we are interested in the on homology, it suffices to work with $\frg[1]$ instead of $\bbT \rmB G$.
The only non-trivial case where we have to check its non-degeneracy is for two tangent vectors of degree one, i.e.~$\xi, \xi' \in \frg$.
There, we obtain
\begin{align}
	\omega(\xi, \xi') &= \omega_{2|(e,e)}(\sigma_{1*}\xi, \sigma_{0*} \xi') - \omega_{2|(e,e)}(\sigma_{0*}\xi, \sigma_{1*} \xi')
	\\*
	&= \omega_{2|(e,e)} \big( (\xi,0), (0,\xi') \big) - \omega_{2|(e,e)} \big( (0,\xi), (\xi', 0) \big)
	\\*
	&= \<\xi, \xi'\> - \<0,0\>\,,
\end{align}
where in the last step we have used the explicit form~\eqref{eq:omega_2 on BG} of $\omega_2$.
Since the Killing form $\<-,-\>$ on $\frg$ is non-degenerate, it follows that $\omega$ is a 2-shifted symplectic form on $\rmB G$.
(This example is by no means new; it can be found in~\cite{PTVV:Shifted_symplectic_structures, Safronov--Quasi-Ham_reduction_via_classical_CSThy, Getzler:Slides_on_Stacks}, for instance.)
\qen
\end{example}

\begin{remark}
Observe the crucial difference from the 2-plectic point of view:
in the 2-shifted symplectic case, the 2-form $\omega_2$ is responsible for the non-degeneracy, whereas in the 2-plectic case it is the 3-form $\omega_3$ on $G$.
The role of $\omega_3$ in the 2-shifted symplectic setting is completely different:
it is purely to establish the (derived) closedness of $\omega_2$.
\qen
\end{remark}

The reason we have included the shifted symplectic perspective is that gerbes provide a promising tool for geometric quantisation in this context as well.
This extends~\cite{LGX--Pre-quasi-sym_quant_via_Grbs} and follows Safronov's recent article~\cite{Safronov:Shifted_GeoQuan}, which also proposes (higher) gerbes as a replacement of line bundles in shifted geometric quantisation.
We propose the following definition, adapted from~\cite{Safronov:Shifted_GeoQuan}:

\begin{definition}
Let $(X,\omega = (\omega_3, \omega_2))$ be a 1-shifted symplectic manifold.
A \emph{1-shifted prequantisation} of $(X,\omega)$ is a triple $(\CG, \CE, \psi)$ of a gerbe $\CG \in \Grb^\nabla(X_0)$ with $\curv(\CG) = \omega_3$, an isomorphism $\CE \colon d_1^*\CG \to d_0^*\CG$ over $X_1$ with $\curv(\CE) = \omega_2$ and a parallel 2-isomorphism $\psi \colon d_0^*\CE \circ d_2^* \CE \to d_1^*\CE$ over $X_2$, which satisfies an associativity condition over $X_3$.
\end{definition}

This provides prequantisations for quasi-symplectic groupoids even when $\omega_3$ is not exact, thus circumventing the exactness constraint in~\cite{LGX--Pre-quasi-sym_quant_via_Grbs}.
If $X = M \dslash G$ for some action of a Lie group $G$ on a manifold $M$ (cf.~Example~\ref{eg:M//G as spl Mfd}), the data $(\CG, \CE,\psi)$ are precisely an equivariant gerbe with connection as defined in~\cite{BMS:Smooth_2Grp_Ext_and_GrbSym}, whose curvatures coincide with $\omega$.

For $p$-shifted symplectic simplicial manifolds with $p > 1$, we would, in general, have to pass to higher gerbes in order to prequantise these simplicial manifolds.
However, in the case of the 2-shifted symplectic simplicial manifold $(\rmB G, \omega)$ from Example~\ref{eg:BG as 2-shifted symplectic} we are lucky:
since $\rmB G_0 = *$, any higher gerbe on $\rmB G_0$ is necessarily trivial (see Definition~\ref{def:n-gerbe} and Proposition~\ref{st:diffcoho hexagon}), and we can define:

\begin{definition}
A \emph{2-shifted prequantisation of $(\rmB G, \omega)$} is a triple $(\CG, \CE, \psi)$ of a gerbe $\CG \in \Grb^\nabla(\rmB G_1)$ with $\curv(\CG) = \omega_3$, an isomorphism $\CE \colon d_2^*\CG \otimes d_0^*\CG \to d_1^*\CG$ over $\rmB G_2$ with $\curv(\CE) = \omega_2$, and a parallel 2-isomorphism $\psi \colon d_1^*\CE \circ d_3^*\CE \to d_2^*\CE \circ d_0^*\CE$ over $\rmB G_3$, satisfying a further coherence condition over $\rmB G_4$.
\end{definition}

We can identify such 2-shifted prequantisations of $(\rmB G, \omega)$ as certain known structures for gerbes, which have not yet been connected with the theory of shifted geometric quantisation:

\begin{theorem}
Let $G$ be a compact, simple, simply connected Lie group with 2-shifted symplectic form $\omega$ as in Example~\ref{eg:BG as 2-shifted symplectic}.
Then, a 2-shifted prequantisation of $(\rmB G, \omega)$ is precisely the same as a multiplicative bundle gerbe as defined and shown to exist in~\cite{Waldorf:Multiplicative_Gerbes}.
In particular, $(\rmB G, \omega)$ admits a 2-shifted prequantisation by~\cite[Ex.~1.5]{Waldorf:Multiplicative_Gerbes}.
\end{theorem}

\begin{appendix}

\def\theequation{\thesection.\arabic{equation}}    
\def\theproposition{\thesection.\arabic{proposition}}    
\def\thedefinition{\thesection.\arabic{definition}}    
\def\thelemma{\thesection.\arabic{lemma}}    
\def\theremark{\thesection.\arabic{remark}}

\section{A glance at 2-categories}
\label{app:2Cats}

We give a very brief overview of basic notions of \emph{2-categories}, or \emph{bicategories}.
(We warn the reader that we use these terms interchangeably here, which is not standard; bicategories are often understood to be the more general concept, where 2-categories are \emph{strict} bicategories.)
We attempt in no way to be complete here; we refer readers interested in full definitions and more background to~\cite{Leinster:Basic_BiCats} for a concise introduction, and to~\cite{SP--Thesis} for a detailed and comprehensive account of 2-categories, including symmetric monoidal 2-categories.

\begin{definition}
A \emph{2-category $\scC$} consists of
\begin{myitemize}
\item a collection of objects, for which we write $x \in \scC$,

\item for each pair of objects $x,y \in \scC$ a \emph{morphism category} $\scC(x,y) = \Hom_\scC(x,y)$, whose objects are called \emph{(1-)morphisms} $f \colon x \to y$, and whose morphisms $\psi \colon f \to g$ are called \emph{2-morphisms} (the composition in $\Hom_\scC(x,y)$ is called \emph{vertical composition}, and we denote it by $(-) \bullet (-)$),

\item for each $x,z,y \in \scC$ a \emph{composition functor} $(-) \circ (-) \colon \Hom_\scC(y,z) \times \Hom_\scC(x,y) \to \Hom_\scC(x,z)$, 

\item for each $x \in \scC$ a specified \emph{identity morphism} $1_x \in \Hom_\scC(x,x)$, and

\item natural isomorphisms
\begin{align}
	\alpha_{f,g,h} \colon (h \circ g) \circ f
	&\arisom h \circ (g \circ f)\,,
	\\
	\lambda_g \colon 1_y \circ g &\arisom g\,,
	\\
	\rho_g \colon g \circ 1_x &\arisom g\,,
\end{align}
for all morphisms $h \colon y \to z$, $g \colon x \to y$, and $f \colon w \to x$ in $\scC$.
The natural isomorphisms $\alpha$, $\rho$, and $\lambda$ are called the \emph{associator} and \emph{left} and \emph{right unitor}, respectively.
\end{myitemize}
These data have to satisfy the \emph{pentagon} and \emph{triangle axioms}, which are, respectively, the commutativity of the following diagrams:
\begin{equation}
\begin{tikzcd}[column sep=-0.5cm, row sep=1cm]
	& &
	\big( (k \circ h) \circ g \big) \circ f \ar[dll, "\alpha_{g,h,k} \circ 1_f"'] \ar[drr, "\alpha_{f,g,k \circ h}"]
	& &
	\\
	\big( k \circ (h \circ g) \big) \circ f \ar[dr, "\alpha_{f, h \circ g, k}"']
	& & & &
	(k \circ h) \circ (g \circ f) \ar[dl, "\alpha_{f \circ g, h, k}"]
	\\
	& k \circ \big( (h \circ g) \circ f \big) \ar[rr, "1_k \circ \alpha_{f, g, h}"']
	& &
	k \circ \big( h \circ (g \circ f) \big)
\end{tikzcd}
\end{equation}
\begin{equation}
\begin{tikzcd}
	(g \circ 1_x) \circ f \ar[rr, "\alpha_{f, 1_x, g}"] \ar[dr, "\rho_g \circ 1_f"']
	& & g \circ (1_x \circ f) \ar[dl, "1_g \circ \lambda_f"]
	\\
	& g \circ f &
\end{tikzcd}
\end{equation}
\end{definition}

Note that because of the functoriality of the composition the \emph{interchange law}
\begin{equation}
	(\psi' \bullet \psi) \circ (\varphi' \bullet \varphi)
	= (\psi' \circ \varphi') \bullet (\psi \circ \varphi)
\end{equation}
holds true for any collection of 2-morphisms for which either side is defined.

\begin{definition}
\label{def:invertible mp in 2Cat}
A morphism $f \colon x \to y$ in a 2-category $\scC$ is called \emph{invertible} if there exists a morphism $g \colon y \to x$ and 2-isomorphisms $1_x \to g \circ f$ and $f \circ g \to 1_y$.
\end{definition}

\begin{example}
The collection of categories naturally assembles into a 2-category $2\Cat$:
its objects are the categories, and for categories $\sfC, \sfD$, the morphism category $\Hom_{2\Cat}(\sfC, \sfD)$ is the category of functors $F \colon \sfC \to \sfD$, with natural transformations $\eta \colon F \to G$ as morphisms.
In this case, the associator and unitors happen to be identity morphisms; one says that $2\Cat$ is a \emph{strict 2-category}.
\qen
\end{example}

\begin{example}
The collection of gerbes (resp.~gerbes with connection) on a manifold $M$ form a 2-category $\Grb(M)$ (resp.~$\Grb^\nabla(M)$); see Section~\ref{sec:Grbs and twVBuns}.
Here, composition of morphisms relies on forming pullbacks of vector bundles and surjective submersions.
This operation is not \emph{strictly} compatible with composition; for two smooth maps $g \colon M'' \to M'$ and $f \colon M' \to M$, and a vector bundle $E \to M$, the pullback bundles $g^*f^*E$ and $(f \circ g)^*E$ are not equal, but there exists a \emph{canonical} isomorphism between them, natural in $E$.
These isomorphisms induce the associator in $\Grb(M)$ (resp.~$\Grb^\nabla(M)$).
\qen
\end{example}

One can also define \emph{monoidal 2-categories}, which are 2-categories $\scC$ endowed with the additional data of a tensor product 2-functor $\otimes \colon \scC \times \scC \to \scC$, together with various 1- and 2-isomorphisms which establish its associativity and unitality.
Further, there is a hierarchy of different levels of commutativity for such products;
each of these levels corresponds to further choices of isomorphisms and coherence conditions.
Writing out these data and conditions requires considerable work; a full treatment can be found in~\cite[Appendix~C]{SP--Thesis}.
For the symmetric monoidal 2-categories $(\Grb(M), \otimes)$ and $(\Grb^\nabla(M), \otimes)$, the tensor product is constructed from pullbacks of submersions and bundles, as well as the tensor product of vector bundles (see Section~\ref{sec:operations on BGrbs}).
Therefore, all additional coherence data arise as the standard \emph{canonical} isomorphisms which relate different ways of pulling back the same geometric structures and which establish the associativity of the tensor product of vector bundles.

\end{appendix}

\begin{small}

\makeatletter

\interlinepenalty=10000

\makeatother

\bibliographystyle{alphaurl}
\addcontentsline{toc}{section}{References}
\bibliography{Gerbes_Survey_Bib}

\vspace{0.5cm}

\noindent
(Severin Bunk)
Universität Hamburg, Fachbereich Mathematik, Bereich Algebra und Zahlentheorie,
\\
Bundesstraße 55, 20146 Hamburg, Germany
\\
severin.bunk@uni-hamburg.de
\end{small}

\end{document}